\documentclass[10pt, reqno]{amsart}

\usepackage{fullpage}
\usepackage[mathscr]{eucal}
\usepackage[all]{xy}
\usepackage{mathrsfs}
\usepackage{xypic}
\usepackage{amsfonts}
\usepackage{amsmath}
\usepackage{amsthm}
\usepackage{amssymb}
\usepackage{latexsym}
\usepackage{eufrak}

\theoremstyle{plain}
\newtheorem{theorem}[equation]{Theorem}
\newtheorem{proposition}[equation]{Proposition}
\newtheorem{lemma}[equation]{Lemma}
\newtheorem{corollary}[equation]{Corollary}

\theoremstyle{definition}
\newtheorem{definition}[equation]{Definition}

\theoremstyle{remark}
\newtheorem{example}[equation]{Example}
\newtheorem{remark}[equation]{Remark}

\newtheorem{notation}[equation]{Notation}

\makeatletter
\renewcommand{\subsection}{\@startsection{subsection}{2}{0pt}{-3ex
plus -1ex minus -0.2ex}{-2mm plus -0pt minus
-2pt}{\normalfont\bfseries}} \makeatother

\numberwithin{equation}{subsection}

\newcommand{\Lmod}[1]{#1\text{-}{\mathsf{mod}}}

\newcommand{\idot}{{\:\raisebox{1pt}{\text{\circle*{1.5}}}}}
\DeclareMathOperator{\Quot}{{\mathrm{rep}}}

\DeclareMathOperator{\Ext}{\mathrm{Ext}}
\DeclareMathOperator{\Hom}{\mathrm{Hom}}
\DeclareMathOperator{\End}{\mathrm{End}}
\DeclareMathOperator{\Der}{\mathrm{Der}}
\DeclareMathOperator{\im}{\mathrm{Im}}
\DeclareMathOperator{\Ker}{\mathrm{Ker}}

\DeclareMathOperator{\rep}{{\mathrm{rep}}}
\newcommand{\rrep}{{\rep}^n_C }

\DeclareMathOperator{\gr}{\mathtt{gr}}

\DeclareMathOperator{\Sym}{\mathrm{Sym}}
\DeclareMathOperator{\Supp}{\mathtt{Supp}}
\newcommand{\supp}{\varpi}

\DeclareMathOperator{\Ad}{\mathrm{Ad}}

\DeclareMathOperator{\Lie}{\mathrm{Lie}}

\DeclareMathOperator{\ad}{\mathrm{ad}}

\newcommand{\dis}{\displaystyle}
\newcommand{\erem}{\hfill$\lozenge$\end{remark}}
\newcommand{\beq}{\begin{equation}\label}
\newcommand{\eeq}{\end{equation}}
\newcommand{\en}{\enspace}
\renewcommand{\o}{\otimes}
\newcommand{\wh}{\widehat}
\newcommand{\wt}{\widetilde}

\renewcommand{\det}{\oper{\mathtt{det}}}
\newcommand{\f}[1]{\mathfrak{#1}}
\newcommand{\scr}[1]{\mathscr{#1}}
\newcommand{\oper}{\operatorname}
\newcommand{\De}{\Delta }
\newcommand{\Om}{\Omega }
\newcommand{\act}{{\mathtt{act}}}

\DeclareMathOperator{\Spec}{\mathrm{Spec}}

\newcommand{\iso}{{\,\stackrel{_\sim}{\rightarrow}\,}}

\def\ccirc{{{}_{^{\,^\circ}}}}

\newcommand{\sminus}{\smallsetminus}
\renewcommand{\mid}{\enspace\big|\enspace}

\newcommand{\too}{\,\,\longrightarrow\,\,}

\newcommand{\mto}{\mapsto}
\newcommand{\map}{\longrightarrow}
\newcommand{\onto}{\twoheadrightarrow}
\newcommand{\cd}{\!\cdot\!}

\newcommand{\pb}{\noindent$\bullet\quad$\parbox[t]{140mm}}

\def\hp{\hphantom{x}}

\newcommand{\hhh}{ Harish-Chandra }

\newcommand{\bdel}{{\boldsymbol{\delta}}}
\newcommand{\ssl}{{\scr L}}
\newcommand{\vi}{${\sf {(i)}}\;$}
\newcommand{\vii}{${\sf {(ii)}}\;$}
\newcommand{\viii}{${\sf {(iii)}}\;$}
 
\newcommand{\syn}{{{\mathbb{S}}_n}}
\newcommand{\g}{{\mathfrak{g}}}

\newcommand{\ff}{{\mathbf{f}}}

\newcommand{\Ga}{\Gamma }
\newcommand{\inv}{^{-1}}
\newcommand{\oo}{{\mathcal O}}

\newcommand{\wg}{\wt{SL}_n }
\renewcommand{\part}{{\f P}}

\newcommand{\sv}{{\mathfrak{s}\mathfrak{l}}}
\newcommand{\mnil}{{\mathscr{M}}_{\mathsf{nil}}}

\newcommand{\arr}{\overset{{\,}_\to}}
\newcommand{\rad}{\overset{{\,}_\circ}}

\newcommand{\BT}{{\mathbb T}}
\newcommand{\bt}{{\mathfrak t}}

\newcommand{\ld}{\ldots }

\def\C{{\mathbb{C}}}

\def\gl{{\mathfrak{g}\mathfrak{l}}}

\newcommand{\e}{{\mathsf{e}}}
\newcommand{\h}{{\mathfrak{h}}}
\newcommand{\sset}{\subset}

\newcommand{\cyc}{^{\oper{cyc}}}

\newcommand{\D}{{\scr D}}
\renewcommand{\P}{{\mathbb{P}}}

\newcommand{\X}{{\f X}}
\newcommand{\Xr}{{\f X}_{n,m}^{\operatorname{reg}}}
\newcommand{\wtxr}{{\wt{\f X}}_{n,m}^{\operatorname{reg}}}

\newcommand{\into}{{\,\hookrightarrow\,}}

\newcommand{\Z}{{\mathbb{Z}}}

\newcommand{\Id}{\operatorname{Id}}

\newcommand{\la}{{\lambda}}

\newcommand{\bx}{{\mathbf x}}
\newcommand{\trig}{{\operatorname{trig}}}

\newcommand{\BH}{{\mathbb H}}
\newcommand{\sH}{{\mathsf H}}
\newcommand{\oH}{{\overline{H}}}
\newcommand{\fx}{{\mathfrak x}}
\newcommand{\fy}{{\mathfrak y}}
\newcommand{\sy}{{\mathsf y}}
\newcommand{\BA}{{\mathbb A}}
\newcommand{\bD}{{\mathbf D}}

\newcommand{\CK}{{\mathcal K}}
\newcommand{\CE}{{\mathcal E}}

\newcommand{\sA}{{\mathsf A}}
\newcommand{\CB}{{\mathcal B}}
\newcommand{\CF}{{\mathcal F}}
\newcommand{\CG}{{\mathcal G}}

\newcommand{\CH}{{\mathcal H}}

\newcommand{\CL}{{\mathcal L}}

\newcommand{\BCL}{{\boldsymbol{\mathcal L}}}

\newcommand{\CN}{{\mathcal N}}
\newcommand{\CO}{{\mathcal O}}
\newcommand{\CU}{{\mathcal U}}

\newcommand{\BC}{{\mathbb C}}

\newcommand{\BO}{{\mathbb O}}

\newcommand{\BZ}{{\mathbb Z}}

\newcommand{\Vo}{V^\circ }
\newcommand{\Nnil}{{\mathbb{M}}_{\mathsf{nil}}}

\newcommand{\sL}{{\mathsf L}}

\newcommand{\bN}{{\mathbf U}}

\newcommand{\gen}{^{\heartsuit\!}}

\newcommand{\reg}{{\operatorname{reg}}}

\begin{document}

\title{{\textbf{
Cherednik algebras
for algebraic curves}}}

\author{\sc{Michael Finkelberg}}
\address{Institute for Information Transmission Problems, and 
Independent Moscow University, Bolshoj Vlasjevskij Pereulok, dom 11,
Moscow 119002 Russia}
\email{fnklberg@gmail.com}

\author{\sc{Victor Ginzburg}}
\address{Department of Mathematics, University of Chicago,
Chicago, IL 60637, USA}
\email{ginzburg@math.uchicago.edu}

\maketitle

\centerline{\em To George Lusztig with admiration}
\vskip 10pt

\begin{abstract}  For any algebraic curve $C$ and $n\geq 1$,
P. Etingof introduced  a `global'  Cherednik algebra
as a natural deformation of the cross product 
$\D(C^n)\rtimes \syn,$
 of the algebra of 
 differential operators on $C^n$ and the
symmetric group. We provide a construction
of the global Cherednik algebra in terms of quantum
Hamiltonian reduction. We study a category of {\em character
$\D$-modules} on a representation scheme associated to $C$
and define a Hamiltonian reduction functor from that
category to  category $\CO$ for the global
 Cherednik algebra.

In the special case of the curve $C=\C^\times$, the global Cherednik algebra
reduces to
the trigonometric Cherednik algebra
of type $\mathbf{A}_{n-1},$ and our  character
$\D$-modules become  holonomic $\D$-modules on
$GL_n(\C)\times \C^n$. The corresponding
perverse sheaves  are reminiscent of (and include as special cases)
Lusztig's 
{\em character sheaves}. 
\end{abstract}

\bigskip

\centerline{\sf Table of Contents}
\vskip -1mm

$\hspace{20mm}$ {\footnotesize \parbox[t]{115mm}{
\hp${}_{}$\!\hp1.{ $\;\,$} {\tt Introduction}\newline
\hp2.{ $\;\,$} {\tt A representation scheme}\newline
\hp3.{ $\;\,$} {\tt Cherednik algebras associated to algebraic curves}\newline
\hp4.{ $\;\,$} {\tt Character sheaves}\newline
\hp5.{ $\;\,$} {\tt The trigonometric case}
}}

\section{Introduction} 
\subsection{}
Associated with an integer $n\geq 1$ and an  algebraic curve $C$, there is
an interesting family, $\sH_{\kappa,\psi}$, of sheaves of associative
algebras on $C^{(n)}=C^n/\syn,$ the $n$-th symmetric power of $C$.
The algebras in question, referred to as
{\em global Cherednik algebras}, see \S\ref{chered},  are natural
 deformations of the cross-product $\D_\psi(C^n)\rtimes\syn$,
of the sheaf of (twisted) differential operators 
\footnote{We refer the reader to \cite{BB2} and \cite{K} for the
basics 
of the theory of twisted differential operators.}
on $C^n$ and the
symmetric group $\syn$ that acts on $\D_\psi(C^n).$
The algebras  $\sH_{\kappa,\psi}$ 
 were introduced by P. Etingof, \cite{Eti}, as `global counterparts' of
rational Cherednik algebras studied in \cite{EG}.

The global Cherednik algebra $\sH_{\kappa,\psi}$
contains an important {\em spherical subalgebra}
$\e\sH_{\kappa,\psi}\e$, where $\e$ denotes
the symmetriser idempotent in the group algebra
of the group $\syn$.
We generalize \cite{GG}, and prove that the algebra
$\e\sH_{\kappa,\psi}\e$ may
be obtained as a quantum Hamiltonian reduction of
$\D_{n\kappa,\psi}(\rep^n_C\times \P^{n-1})$, a sheaf of twisted
 differential operators on  $\rep^n_C\times \P^{n-1}$,
cf. Theorem \ref{sheaf eg}.

Our result provides a strong link between
categories of 
${\D_{n\kappa,\psi}(\rep^n_C\times \P^{n-1})}$-modules
and ${\sH_{\kappa,\psi}}$-modules, respectively.
Specifically, following the strategy of
\cite{GG}, \S7, we construct
 an exact functor
\begin{equation}\label{bh}
\BH:\ \Lmod{\D_{n\kappa,\psi}(\rep^n_C\times \P^{n-1})}\too\Lmod{\sH_{\kappa,\psi}},
\end{equation}
called the functor of Hamiltonian reduction.

\subsection{}\label{int1} In mid 80's, G. Lusztig introduced an important notion 
of  {\em character sheaf} on a  reductive
algebraic group $G$. In more detail,
write $\g$ for the Lie algebra of $G$ and  use the Killing form to
identify $\g^*\cong\g$. Let
$\CN\sset \g^*$ be the image of the set of nilpotent elements
in $\g,$ and let $G\times\CN\sset  G\times\g^*=T^*G$
be the {\em nil-cone} in the total space of the  cotangent bundle on $G$.

Recall further that, associated
with any perverse sheaf $M$ on $G$, one has its characteristic variety
$SS(M)\sset T^*G$. A character sheaf is, by definition,
an $\Ad G$-equivariant perverse sheaf
$M,$ on $G,$ such that the corresponding characteristic
variety  is nilpotent, i.e., such that we have
$SS(M)\sset G\times \CN.$

We will be interested in the special case  $G=GL_n$.
Motivated by the geometric Langlands conjecture,
G. Laumon, \cite{La}, generalized the notion of character
sheaf on $GL_n$ to the `global setting' involving an arbitrary
  smooth
algebraic curve. Given such a curve $C$,
Laumon replaces the  adjoint quotient stack $G/\Ad G$ by  $Coh^n_C,$ a
certain stack  of length $n$ coherent sheaves on $C.$
He then  defines a  {\em global
nilpotent subvariety} of the cotangent bundle
$T^*Coh^n_C,$ cf. \cite{La2}, and considers the class of
 perverse sheaves $M$
on  $Coh^n_C$ such that $SS(M)$ is contained in the  global
nilpotent subvariety.

In this paper, we introduce character sheaves on
$\rep^n_C\times\P^{n-1}$. Here, the scheme $\rep^n_C$ is
 an appropriate
Quot scheme of length $n$  sheaves on $C$,
a close cousin of  $Coh^n_C$, and   $\P^{n-1}$
is  an $n-1$-dimensional projective space.
In section \ref{lagrangian}, we define 
a   version of
`global nilpotent variety' $\Nnil\sset T^*(\rep^n_C\times\P^{n-1})$,
 and introduce a class of
 $\D$-modules
on  $\rep^n_C\times \P^{n-1}$,  called
 {\em character $\D$-modules}, which have a  nilpotent characteristic
 variety,
i.e., are such that  $SS(M)\sset\Nnil$, see Definition
\ref{characterDmod}.

The group  $G=GL_n$ acts on both
$\rep^n_C$ and  $\P^{n-1}$ in a natural way.
In  analogy with
the theory studied by Lusztig and Laumon,
perverse  sheaves associated to   character $\D$-modules
via the Riemann-Hilbert correspondence
 are
locally constant along  $G$-diagonal
orbits in  $\rep^n_C\times \P^{n-1}$.

A very special feature of the $G$-variety  $\rep^n_C\times \P^{n-1}$
is that the corresponding nilpotent variety, $\Nnil,$ turns out to be
a  {\em Lagrangian} subvariety in
$T^*(\rep^n_C\times\P^{n-1})$. This follows from a
geometric result saying that the
group $GL_n$ acts diagonally on $\CN\times \C^n$ with {\em finitely
many} orbits, \cite{GG}, Corollary 2.2. These orbits may be
parametrised by the pairs $(\lambda,\mu),$ of arbitrary partitions
$\lambda=\lambda_1+\ldots+\lambda_p$ and
$\mu=\mu_1+\ldots+\mu_q$, with total sum
$\lambda+\mu=n$, see \cite{AH} and  \cite{T}.

\subsection{}
Character  $\D$-modules play an important role
in representation theory of the global Cherednik  algebra ${\sH_{\kappa,\psi}}$.
In more detail, there is a natural analogue,
$\CO({\sH_{\kappa,\psi}}),$ of the Bernstein-Gelfand-Gelfand category $\CO$ for
the  global Cherednik algebra, see Definition \ref{catO}.
We show (Proposition \ref{ham_fun}) that the Hamiltonian reduction
functor \eqref{bh}
sends character $\D$-modules to objects of the category $\CO({\sH_{\kappa,\psi}}),$
moreover, the latter category gets identified,
via the functor $\BH$, with a quotient of the
former by the Serre subcategory $\Ker\BH$.

In  the special case
of the curve $C=\C^\times$, the global Cherednik
algebra
reduces to $\sH_{\kappa}$, the {\em trigonometric} Cherednik algebra
of type ${\mathbf{A}}_{n-1},$ see \S\ref{Trig} for definitions,
and the sheaves
considered by Laumon become
 Lusztig's character sheaves on the group $GL_n$.
Similarly, our 
character $\D$-modules become (twisted)  $\D$-modules on
$GL_n\times\P^{n-1}$.

Given  a character sheaf on $GL_n$ in the sense of Lusztig,
one may pull-back the corresponding $\D$-module via the first projection
$GL_n\times \P^{n-1}\to GL_n.$
The resulting $\D$-module on $GL_n\times \P^{n-1}$ is a character $\D$-module
in our sense. However,  there are many other quite interesting
character $\D$-modules on $GL_n\times \P^{n-1}$ which do not come from
Lusztig's character sheaves on~$GL_n.$

Sometimes, it is more convenient to replace $GL_n$ by its subgroup $SL_n$.
In that case, we prove  that
{\em cuspidal} character $\D$-modules
correspond, via the Hamiltonian 
functor, to finite
dimensional representations of the corresponding
Cherednik algebra, cf. Corollary \ref{findimcor}
and Theorem \ref{cusp_thm}.

\vskip -5pt
\subsection{Convention.} The trigonometric Cherednik algebra
depends on one complex parameter, to be denoted $\kappa\in\C$.
Such an algebra corresponds, via the
quantum Hamiltonian reduction construction,
to a  sheaf of twisted differential operators  on $SL_n\times\P^{n-1}$,
which is also   labelled by one complex parameter, to be denoted
 $c\in\C$.
Throughout the paper, we will use the normalization of the
above parameters $\kappa$ and $c$, such that
the sheaf of TDO with parameter $c$ gives rise to
the  trigonometric Cherednik algebra with parameter
\begin{eqnarray}\label{norm}
\kappa=c/n.
\end{eqnarray}
\subsection{Acknowledgements}{\footnotesize{
We are grateful to Pavel Etingof and Roman Travkin for useful discussions.
M.F. is grateful to the University of Chicago and Indiana University
at Bloomington for the hospitality and support. He was partially
supported by the CRDF award RUM1-2694, and the ANR program `GIMP',
contract number ANR-05-BLAN-0029-01. 
The work of V.G. was  partially supported by the NSF grant DMS-0601050.}}
\pagebreak[3]

\section{A representation scheme}

\subsection{}
\label{nezabudki} 
In this paper, we work over $\C$ and
 write $\Hom=\Hom_\C,\,\End=\End_\C,\,\o=\o_\C.$

Let $Y$ be a scheme  with
structure sheaf $\oo_Y$  and coordinate ring
~$\CO(Y).$

\begin{definition} Let $V$ be a finite
dimensional vector space, $\CF$ a  finite
 length (torsion) $\CO_Y$-sheaf, and
$\vartheta: V\iso\Gamma(Y,\CF),$  a vector space isomorphism.
Such a data $(V,\CF,\vartheta)$ is called
a {\em representation} of $\CO_Y$ in $V.$
\end{definition}

\begin{proposition}\label{rep} \vi For each $n\geq 1,$ there is a $GL_n(\C)$-scheme
$\rep^n_Y$, of finite type, that parametrizes representations
of $\oo_Y$ in $\C^n.$

\vii If $Y$ is affine then we have
$\rep^n_Y=\rep^n \CO(Y)$, is the affine scheme that parametrizes
algebra homomorphisms $\CO(Y)\to\End \C^n$.
\end{proposition}
\begin{proof} The scheme $\rep^n_Y$
may be identified with
an open subscheme of  Grothendieck's Quot-scheme that
parametrizes   surjective morphisms
$\C^n\otimes\CO_Y\twoheadrightarrow\CF$ such that the
composition $\C^n\hookrightarrow\Gamma(Y,\C^n\otimes\CO_Y)$
$\onto\Gamma(Y,\CF)$
is an isomorphism. 
The group $GL_n$ acts on $\rep^n_Y$ by base change transformations,
that is, by changing the isomorphism $\C^n\cong\Gamma(Y,\CF)$.
\end{proof}

Below, we fix $n\geq 1$, let $V=\C^n,$ and $GL(V)=GL_n(\C)$. 
Given a representation $(\CF, \vartheta)$ in $V$,
the isomorphism
$\vartheta: V\iso\Gamma(Y,\CF)$ 
makes the vector space $V$ an
$\CO(Y)$-module.
We often drop 
$\vartheta$ from the notation and
 write $\CF\in \rep^n_Y$.

\begin{example}\label{C=C} For $Y:=\BA^1$, we have $\CO(Y)=\C[t]$, hence
$\rep^n_Y=\rep^n\C[t]=\End V$. 

 For $Y:=\C^\times$, we have $\CO(Y)=\C[t,t\inv]$, hence
$\rep^n_Y=\rep^n\C[t,t\inv]= GL(V)$. 
In these examples, the
action of the
group $GL(V)$ is  the conjugation-action.
\end{example}

\subsection{Factorization property.}
From now on, we will be concerned
exclusively with the case where $Y=C$ 
is a  smooth  algebraic curve
(either affine or complete).

Let $\syn$ denote the symmetric group.
Write $C^n$ and $C^{(n)}=C^n/\syn$ for the $n$-th cartesian and
symmetric power of $C$, respectively.
Taking support cycle of a length $n$
coherent  sheaf on $C$ gives a natural
map
\begin{equation}\label{supp}
\supp:\ \rep^n_C\too C^{(n)},\quad
\CF\mto \Supp\CF.
\end{equation}

It is straightforward to see that the map $\supp$ is
an affine and surjective morphism of schemes and that
the group $GL(V)$ acts along the fibers of $\supp$.

The collection of morphisms \eqref{supp} for
various values of the integer $n$ enjoy an important
 {\em factorization property}. Specifically, let $k,m,$ be
a pair of  positive integers  and
 $D_1\in C^{(k)},\ D_2\in C^{(m)}$, $|D_1|\cap |D_2|=\emptyset,$
a pair of divisors with disjoint
supports, that is, a pair
of unordered collections of $k$  and  $m$ points
of $C$, respectively, which
have no points in common.
The factorization property says that there is a natural isomorphism
$$\supp^{-1}(D_1+D_2)\cong
GL_{k+m}\overset{GL_k\times GL_m}{\times}
\big(\supp^{-1}(D_1)\times\supp^{-1}(D_2)\big).
$$

The factorization property also holds in families. 
To explain this, let
${\gen}C^{(k,m)}\sset C^{(k)}\times C^{(m)}$ be the open 
subset formed by all
pairs of divisors with disjoint support. There is a 
natural composite projection ${\gen}C^{(k,m)}\into C^{(k)}\times C^{(m)}\onto C^{(k+m)}$. 

The family version of the factorization property reads
\begin{equation}\label{fac}
\Quot^{k+m}_C\underset{C^{({k+m})}}\times {\gen}C^{(k,m)}\simeq
\left(GL_{k+m}\overset{GL_k\times GL_m}{\times}(\Quot^k_C\times\Quot^m_C)
\right)\underset{{C^{(k)}\times C^{(m)}}}\times  {\gen}C^{(k,m)}.
\end{equation}

\subsection{Tangent and cotangent bundles.}
Let $\Om_C$ be the sheaf of K\"ahler differentials on $C$.
Given a finite length sheaf $\CF$ we let
$\CF^\vee:=\CE xt^1(\CF, \Om_C)$, denote the Grothendieck-Serre
dual of $\CF$. Let $\jmath: C\into \bar{C}$ be an open imbedding of $C$
into a complete curve. Then
we have natural isomorphisms
\begin{equation}\label{serre}
\Ga(C, \CF^\vee)=
\Ext^1(\bar{\CF},\Om_{\bar{C}})\cong
\Hom(\oo_{\bar{C}},\bar{\CF})^*=
\Ga(C,\CF)^*,
\end{equation}
where we write $\bar{\CF}=\jmath_*\CF$ and
where the  isomorphism in the middle is provided by the
Serre duality. In particular,
for any $\CF\in\rrep$, one has canonical
isomorphisms
\begin{equation}\label{can}
\Ga(C^2,
\CF\boxtimes\CF^\vee)=
\Ga(C,\CF)\o \Ga(C,\CF^\vee)\stackrel{\vartheta}\iso
V\o V^*=\End V.
\end{equation}

 Let $\Delta\sset C^2:=C\times C$
denote the diagonal divisor. Thus, on $C^2$,
we have a triple
of sheaves $\oo_{{C^2}}(-\De)\into\oo_{{C^2}}
\into\oo_{{C^2}}(\De)$. 
For any $\CF\in\rrep$, let $\CK er_\CF$ denote the
kernel of the composite morphism
$V\o\oo_C\iso \Ga(C,\CF)\o\oo_C \onto \CF.$

\begin{lemma}\label{tangent} The scheme $\rep^n_C$ is smooth.
For the tangent, resp. cotangent,
 space at a point $\CF\in \rrep,$ there are canonical isomorphisms:
\begin{align*}
T_\CF(\rrep)\;&\cong\en\en \Hom(\CK er_{_\CF}, \CF)\en\en\cong\en\en
\Ga\big(C^2, (\CF\boxtimes\CF^\vee)(\Delta)\big),\quad\oper{resp.}\\
T^*_\CF(\rrep)\;&\cong \;\Ext^1(\CF,\CK er_{_\CF}\otimes\Omega_C)
\;\cong\; \Ga\big(C^2,(\CF\boxtimes\CF^\vee)(-\De)\big).
\end{align*}
\end{lemma}
\begin{proof}
It is well known that the Zariski tangent space to the
scheme parametrizing surjective morphisms $V\o\oo_C\onto \CF$
is equal to the vector space $\Hom(\CK er_{_\CF}, \CF)$.
This proves the first isomorphism.
From this, by Serre duality
one obtains  $T^*_\CF(\rrep)\cong \Ext^1(\CF,\CK
er_{_\CF}\otimes\Omega_C)$.

To complete the proof, it suffices
to prove the second formula for the tangent space  $T_\CF(\rrep)$;
the corresponding formula for  $T^*_\CF(\rrep)$
would then follow by duality.

To prove the  formula for $T_\CF(\rrep)$, we may assume that $C$ is affine
and put $A=\oo(C)$. Also, let $V=\Ga(C,\CF)$ and $K=\Ga(C,\CK er_\CF)$.
Thus, we have $\Hom(\CK er_{_\CF}, \CF)=\Hom_A(K,V)$.

The sheaf extension
$\CK er_\CF\into V\o\oo_C \onto \CF$ yields a short exact
sequence of $A$-modules
$$0\to K\too A\o V\stackrel{\act}\too V\to 0,$$
where $\act$ is the action-map.

The action of $A$ on $V$ makes the vector space $\End V$
an $A$-bimodule. We write $\Der_A(A, \End V)$ for
the space of derivations of the algebra $A$ with coefficients
in the $A$-bimodule $\End V$. Further,
given an $A$-module map $f: K\to V$, for any element
$a\in A$ we define a linear
map 
$$\delta_f(a): V\to V,\quad v\mto f(a\o v -1\o av)$$
It is straightforward to see that the assignment
$a\mto \delta_f(a)$ gives a derivation $\delta_f\in \Der_A(A, \End V)$.
Moreover,  this way, one obtains a canonical isomorphism
$$
\Hom_A(K,V)\iso \Der_A(A, \End V),
\quad f\mto \delta_f.
$$

The above isomorphism provides a well known alternative interpretation
of the tangent space to a representation scheme in the form
$T_\CF(\rrep)\cong \Der_A(A, \End V).$

We observe next that, for any
$a\in A$, one has $a\o 1-1\o a\in \Ga(C^2,\oo_{C^2}(-\Delta))
\sset A\o A$.
Further, for any section
$s\in \Ga\big(C^2,  (\CF\boxtimes\CF^\vee)(\Delta)\big),$
the map
$$A\too \Ga\big(C^2,\CF\boxtimes\CF^\vee)=V^*\o V=\End V,
\quad a\mto \xi^s(a):=s\cdot(a\o 1-1\o a)
$$
is easily seen to be a derivation. Moreover,  one shows  that any
derivation has the form $\xi^s$ for a unique section $s$.

Thus, combining all the above, we deduce the desired canonical isomorphisms
$$
T_\CF(\rrep)\cong\Hom_A(K,V)\cong\Der_A(A, \End
V)\cong
\Ga\big(C^2,  (\CF\boxtimes\CF^\vee)(\Delta)\big).
$$

In particular, for any $\CF\in\rrep$, we find that 
$$\dim T_\CF(\rrep)=\dim \Ga\big(C^2,
(\CF\boxtimes\CF^\vee)(\Delta)\big)=\dim(V\o V^*)=n^2,
$$
is independent of $\CF$. This implies that the scheme
$\rrep$ is smooth. 
\end{proof}

\subsection{Example: formal disc.}\label{point} Given a formal power series $f=f(t)\in\C[[t]]$,
we define
the {\em difference-derivative} of $f$ as a power series
$\dis Df(t',t''):=\frac{f(t')-f(t'')}{t'-t''}\in\C[[t',t'']]$.

For any $X\in \End V,$ write
$L_X,\, R_X:\ \End V\to \End V$, for the pair of  linear maps of
left, resp. right, multiplication by $X$ in the algebra $\End V$.
 These   maps commute, hence, for any polynomial $f\in\C[t]$,
there is a well-defined linear operator $Df(L_X,R_X):\ \End V\to \End V.$
For instance, in the special case
where $f(t)=t^m$, we find
$Df(L_X,R_X)(Y)=X^{m-1}Y+X^{m-2}YX+\ldots+XYX^{m-2}+~YX^{m-1}.$

Let $C_x$ be the completion
 of
a smooth curve $C$
at a point $x\in C$,  a formal scheme.
 A choice of local
parameter on $C_x$ amounts to a choice
of algebra isomorphism $\oo(C_x)=\C[[t]].$
Thus, for the corresponding representation schemes, we obtain
$\rep^n\oo(C_x)\cong\rep^n\C[[t]]=\wh{\CN}$, the completion of $\End V$
along the closed subscheme
$\CN\sset \End V$, of nilpotent
endomorphisms.

The tangent space at a
point of $\rep^n\oo(C_x)=\wh{\CN}$ may be therefore identified
with the vector space $\End V$.
More precisely, the tangent bundle
 $T(\rep^n\CO(C_x))$ is the completion of $\End V\times \End V$ along
 $\CN\times\End V$. 
Abusing the notation slightly, we may
write a point in the  tangent bundle as a pair
 $(X,Y)\in \End V\times \End V$, where
$X\in \rep^n\oo(C_x)=\wh{\CN},$ and $Y$ is a tangent vector at $X$.
Thus, associated with such a point $(X,Y)$ 
 and any  formal power series $f\in\C[[t]],$
there is a well-defined linear operator $Df(L_X,R_X)$,
acting on an appropriate completion of $\End V$.

Now,
let $t$ and $u$ be two different local  parameters on the formal scheme $C_x$.
Thus, one can write
$u=f(t)$, for some $f\in\C[[t]]$. We have

\begin{lemma}  The differential of $f$ acts
on the tangent bundle $T(\rep^n\CO(C_x))$ by the formula
$$df: \ (X,Y)\mto \big(f(X),\,Df(L_X,R_X)(Y)\big).
\eqno\Box$$
\end{lemma}

\begin{remark} 
In the special case where $f(t)=e^t$, writing  $\Phi(t):=(e^t-1)/t$, we 
compute
$$\frac{f(t')-f(t'')}{t'-t''}=\frac{e^{t'}-e^{t''}}{t'-t''}=
\frac{e^{t'-t''}-1}{t'-t''}\cd e^{t''}=\Phi(t'-t'')\cd e^{t''}.
$$
Therefore, for the differential
of the exponential map $X\mto \exp X,$ using the lemma
and the notation $\ad X(Y):=XY-YX=(L_X-R_X)(Y)$,
one recovers the standard formula
$$d\exp: \ (X,Y)\mto \big(f(X),\,Df(L_X,R_X)(Y)\big)=
\left(e^X,\,\Phi(\ad X)(Y)\cd e^X\right).\qquad\Box
$$
\end{remark}

\subsection{\'Etale structure.} The assignment $C\mto \rep_C^n$
is a functor. Specifically,
for any morphism of curves $ p:\ C_1\to C_2,$  the push-forward
of coherent sheaves induces a $GL_n$-equivariant
morphism $ \rep p:\
\rep_{C_1}^n\to \rep_{C_2}^n,\ \CF\mapsto p_*\CF.$

The following result says that, for any smooth curve $C$, the
representation scheme $\rep^n_C$ is locally isomorphic, in \'etale topology,
to $\rep^n_{\BA^1}={\mathfrak{gl}}_n,$ the
corresponding scheme for the curve $C=\BA^1$.

\begin{proposition}\label{etale_prop} For any smooth curve $C$ and
$n\geq 1,$ one can find a finite collection of open sub-curves
$C_s\sset C,\, s=1,\ldots,r,$ such that the following holds.

For each $s=1,\ldots,r,$ there is a Zariski open subset $U_s\sset \rep_{C_s}^n,$
and a diagram of morphisms of curves
$$
\xymatrix{
C\;&&\;C_s\;\ar@{_{(}->}[ll]_<>(0.5){q_s}\ar[rr]^<>(0.5){p_s}&&\;\BA^1,}
\quad s=1,\ldots,r,
$$
that gives rise to a diagram of representation schemes
$$
\xymatrix{
\rep_C^n\;&&\;\rep_{C_s}^n\supset U_s\;\ar@{_{(}->}[ll]_<>(0.5){\rep
q_s}\ar[rr]^<>(0.5){\rep p_s}&&\;\rep^n_{\BA^1}={\mathfrak{gl}}_n.}
$$

The above data satisfies, for each $s=1,\ldots,r,$ 
 the following properties:
\vskip 3pt

\pb{The map $q_s$ is an open imbedding, and  $\rep_C^n=(\rep
q_1)(U_1)\cup\ldots \cup(\rep
q_s)(U_s)$ is an open cover;}
\vskip 1pt

\pb{The map $p_s$ is \'etale and the restriction of the morphism
$\rep p_s$ to $U_s$  is \'etale as well.}
\end{proposition}

This proposition is an immediate consequence of
two lemmas below. Fix a smooth curve $C$.

\begin{lemma}
\label{etale}
Given a collection of pairwise distinct points $c_1,\ldots,c_N\in C$ there is
an open subset ${C'}\subset C$ and an \'etale
morphism $ p:\ {C'}\to\BA^1$ such that $c_1,\ldots,c_N\in{C'}$ and all the values
$ p(c_1),\ldots, p(c_N)$ are  pairwise distinct.
\end{lemma}
\begin{proof}
Pick a point $x\in C\sminus\{c_1,\ldots,c_N\}$. Then $C\sminus\{x\}$ is
affine. Choose an embedding $\gamma:\ C\sminus\{x\}\hookrightarrow\BA^m$. Then 
the composition of $\gamma$ with a general linear projection
$\BA^m\to\BA^1$ is the desired $ p$; while ${C'}$ is the complement of
the set of ramification points of $ p$ in $C\sminus\{x\}$.
\end{proof}

Given an \'etale morphism of curves  $ p:{C}\to\BA^1$ define
$$^0{C}^{(n)}:=\{(c_1,\ldots,c_n)\in {C}^{(n)}\en\big|\en
 p(c_i)= p(c_j)\en\oper{iff}\en c_i=c_j,\; i,j=1,\ldots, n\}.
$$
This is clearly a Zariski open subset in ${C}^{(n)}$. Let
$\supp^{-1}(\ ^0{C}^{(n)})$ be its preimage in $\rep_{C}^n,$ cf. \eqref{supp}.
\begin{lemma}
\label{etal} The   morphism below   induced by the map  $ p:{C}\to\BA^1$
is
\'etale,
$$ \rep p:\ \supp^{-1}(\ ^0{C}^{(n)})\to \rep_{\BA^1}^n=\mathfrak{gl}_n.$$
\end{lemma}
\begin{proof} To simplify notation, write
 $U:=\supp^{-1}(\ ^0{C}^{(n)})$ and $p_*:=\rep p.$

We must check that the differential of $ p_*$ on the
tangent spaces is invertible. Let
$\underline{c}=(k_1c_1+\ldots+k_lc_l)\in\ ^0{C}^{(n)}$ (so that
$c_1,\ldots,c_l$ are all distinct). Then $ p(c_1),\ldots, p(c_l)$
are all distinct. If $\CF\in U$, and $\supp\CF=\underline{c}$, then 
$ p_*\CF\in\rep_{\BA^1}^n$, and
$\supp(p_*\CF)=(k_1 p(c_1)+\ldots+k_l p(c_l))$. 
As we know, $T_\CF\rep_{C}^n=\Hom(\CK er_\CF,\CF)$.

Let ${C}_{c_i}$ (resp. $\BA^1_{ p(c_i)}$) stand for the completion 
of ${C}$ at $c_i$ (resp. $\BA^1$ at $ p(c_i)$). 
Let $(\CK er_\CF)_{c_i}$ (resp. $(\CF)_{c_i}$) denote the restriction
of $\CK er_\CF$ (resp. $\CF$) to ${C}_{c_i}$. Then the
restriction to these completions induces an isomorphism
$$\Hom(\CK er_\CF,\CF)=\bigoplus_{1\leq i\leq l}\Hom((\CK
er_\CF)_{c_i},(\CF)_{c_i}).$$
Similarly, we have
$$\Hom(\CK er_{ p_*\CF}, p_*\CF)=\bigoplus_{1\leq i\leq l}\Hom((\CK
er_{ p_*\CF})_{ p(c_i)},( p_*\CF)_{ p(c_i)}).$$

The map $ p$ being \'etale at $c_i$, it identifies $C_{c_i}$ with
$\BA^1_{ p(c_i)}$. Under this identification $(\CF)_{c_i}$ gets
identified with $( p_*\CF)_{ p(c_i)}$, and $(\CK er_\CF)_{c_i}$ gets
identified with $(\CK er_{ p_*\CF})_{ p(c_i)}$. Observe that
the latter identification
is provided by
the differential of the morphism $\rep p$ at the point $\CF\in\rep^n_C$. The lemma follows.
\end{proof}

\section{Cherednik algebras associated to algebraic curves}
\label{cher_sec}

\subsection{Global Cherednik algebras.}
\label{chered}
For any smooth algebraic curve $C$, P.~Etingof defined in~\cite{Eti},~2.19,
a sheaf of Cherednik algebras on the symmetric power $C^{(n)},\,n\geq
1$.

To recall Etingof's definition, for any smooth variety $Y$,  introduce a
length two
complex of sheaves
$\Omega_Y^{1,2}:=[\Omega_Y^1\stackrel{d}\to(\Omega_Y^2)_\text{closed}]$,
concentrated in degrees 1 and 2. The sheaves of algebraic
twisted differential operators (TDO) on $Y$ are known to be parametrized (up to
isomorphism) by elements of $H^2(Y, \Omega_Y^{1,2}),$
the second hyper-cohomology group, cf. eg. \cite{K}.

\begin{remark} 
For an affine curve $C$, we have
$H^2(C,\Omega^{1,2}_C)=0$; for a projective curve $C$,
 we have $H^2(C,\Omega^{1,2}_C)\cong\C,$ where a generator
 is provided  by the first Chern class of a degree 1 line
bundle on $C$.
\erem

Now, let $C$ be a smooth algebraic curve.
Given a class
 $\psi\in H^2(C,\Omega^{1,2}_C)$,
write $\dis\psi^{\boxtimes n}\in H^2(C^n,\Omega^{1,2}_{C^n})^{\syn}$
for the external
product of $n$ copies of $\psi$.
Pull-back via the projection $C^n\onto C^{(n)}$ induces an isomorphism
$\dis H^2(C^{(n)},\Omega_{C^{(n)}}^{1,2})\iso
 H^2(C^n,\Omega_{C^n}^{1,2})^{\syn}\sset  H^2(C^n,\Omega_{C^n}^{1,2})$.
We let $\psi_n\in H^2(C^{(n)},\Omega^{1,2}_{C^{(n)}})$
denote the preimage of the class $\psi^{\boxtimes n}$
via the above  isomorphism.

For any $n\geq 1,\,\kappa\in\BC,$ and  $\psi\in H^2(C,\Omega^{1,2}_C)$,
Etingof defines the sheaf  of global Cherednik algebras, as follows,
cf. \cite{Eti},~2.9.
Let  $\eta$ be  a 1-form  on $C^n$ 
such that
for the  cohomology class  $\eta\in H^2(C^n,\Omega^{1,2}_{C^n})$ 
one has $d\eta=\psi^{\boxtimes n}.$
Further, for any $i,j\in[1,n]$,
let $\Delta_{ij}\subset C^n$  be the corresponding $(ij)$-diagonal, with equal
$i$-th and $j$-th coordinates. Thus,  $\Delta_n=
\cup_{i\neq j} \Delta_{ij}$, is the big diagonal,
and
the image of  $\Delta_n$ under
the projection $C^n\onto  C^{(n)}$ is the  discriminant divisor,
  $\bD\subset C^{(n)}.$

Given  a vector field $v$ on $C^n$, for each  pair $(i,j)$, choose
 a rational
function $f^v_{ij}$ on $C^n$ whose polar part at $\Delta_{ij}$ corresponds to
$v$, as explained in ~{\em loc. cit.}~2.4. Associated to
such a data, Etingof defines  in~\cite{Eti},~2.9 the following
 Dunkl operator 
$$D_v:=\operatorname{Lie}_v+\langle v,\eta\rangle+\kappa\cd n\cd\sum_{i\ne j}(s_{ij}-1)\otimes f^v_{ij}
\,\in\; \D_{\psi_n}(C^n\sminus\Delta_n)\rtimes\syn.
$$
Here $\Lie_v$ stands for the Lie derivative
with respect to the vector field $v$ and  the
cross-product algebra $\D_{\psi_n}(C^n\sminus\Delta_n)\rtimes\syn$
on the right
is viewed as a sheaf of associative algebras on $C^{(n)}\sminus {\mathbf D}$.

The sheaf of Cherednik algebras is defined 
as a subsheaf of the sheaf
$\jmath_\idot\D_{\psi_n}(C^n\sminus\Delta_n)\rtimes\syn,$
where $\jmath: C^{(n)}\sminus {\mathbf D}\into C^{(n)}$ stands
for the
open imbedding.

 Specifically, 
following Etingof, we have

\begin{definition} Let 
 $\sH_{\kappa,\psi_n}$,  the sheaf of Cherednik algebras,  be the
sheaf, on $C^{(n)}$, of associative subalgebras,
 generated by all regular functions on $C^n$ and
by  the Dunkl operators
$D_v$, for all vector fields $v$ and $f_{ij}^v$ as above. 
\end{definition}

Next, let $\e\in\BC[\syn]$ be the idempotent projector to the trivial
representation. The subalgebra 
$\e\sH_{\kappa,\psi_n}\e\sset \sH_{\kappa,\psi_n}$
is called {\em spherical subalgebra}.
This is a sheaf of associative algebras on $C^{(n)}$
that may be identified naturally with
 a subsheaf of $\D_{\psi_n}(C^{(n)})$.

The following is a global analogue of a result due to
Bezrukavnikov and Etingof \cite{BE}, and
Gordon and Stafford \cite{GS}, Theorem  3.3,
 in the case where $C=\BA^1$.

\begin{proposition}
\label{morita} 
For any $\kappa\in\C\sminus [-1,
0),$  
 the  functor below is a Morita equivalence  
$$\Lmod{\sH_{\kappa,\psi_n}}\too
\Lmod{\e\sH_{\kappa,\psi_n}\e},\quad M\mto \e M.
$$
\end{proposition}

\begin{proof} It is a well known fact that the Morita equivalence
statement is equivalent to an equality
$\sH_{\kappa,\psi_n}=\sH_{\kappa,\psi_n}\cd\e\cd\sH_{\kappa,\psi_n}.$
In any case, on $C^{(n)}$, one has an exact sequence of sheaves
\begin{equation}\label{moritaseq}
0\to\sH_{\kappa,\psi_n}\cd\e\cd\sH_{\kappa,\psi_n}\to
\sH_{\kappa,\psi_n}\to
\sH_{\kappa,\psi_n}/\sH_{\kappa,\psi_n}\cd\e\cd\sH_{\kappa,\psi_n}\to 0.
\end{equation}

Proving that the sheaf on the right vanishes is a `local' problem.
Thus, one can restrict \eqref{morita} to an open subset
in  $C^{(n)}$. Then, we are in a position to 
use the above cited result of  Gordon and Stafford
saying that  Proposition \ref{morita} holds for the curve $C=\BA^1$.
In effect, by Proposition \ref{etale_prop},
each point in  $C^{(n)}$
is contained in an open subset $U$ with the following property.
There exists an \'etale morphism 
$f:\ U\to(\BA^1)^{(n)}$
such that the pull-back via $f$ of the sheaf on the right 
of \eqref{morita} for the curve $\BA^1$ is equal to the corresponding sheaf for 
the curve $C$.
\end{proof}

The sheaf $\sH_{\kappa,\psi_n}$  comes equipped with an
increasing filtration  arising from
 the standard filtration by the order of differential
operator, cf. \cite{EG}, \cite{Eti}. 
The filtration on  $\sH_{\kappa,\psi_n}$ induces, by restriction,
an increasing filtration, $F_\idot(\e\sH_{\kappa,\psi_n}\e),$
on the  spherical subalgebra.
Etingof proved a graded
algebra isomorphism, cf. \cite{Eti},
\begin{equation}\label{gralg}
\gr^F(\e\sH_{\kappa,\psi_n}\e)\cong
p_*\CO_Y,\quad Y:=(T^*C)^{(n)}=(T^*(C^n))/\syn,
\end{equation}
where $p: (T^*C)^{(n)}\to C^{(n)}$ denotes the natural projection.

\subsection{The determinant line bundle.} 
\label{detline} We fix a curve $C$, an integer $n\geq 1$,
and let $V=\C^n$.
Set  $X_n:=\rrep\times V,$ a
smooth variety.
Let $X_n^{\oper{cyc}}\subset X_n=\rep^n_C\times  V$ be a subset formed by the
triples $(\CF,\vartheta,v)$ such that $v$ is a cyclic vector,
i.e., such that the morphism $\CO_C \to \CF,\,f\mto f v$ is surjective. 
It is clear that the set
$X_n^{\oper{cyc}}$ is  a $GL(V)$-stable Zariski open subset of $X_n$.

Choose a basis of $V$,
and identify $V=\C^n$, and $GL(V)=GL_n$, etc. Assume further
that our  curve $C$ admits a  {\em global} coordinate $t: C\into\BA^1.$
Then, associated with each pair $(\CF,\vartheta,v)\in X_n^{\oper{cyc}}$,
there is a {\em  matrix}
${\mathsf{g}}(\CF,\vartheta,v)\in GL_n$, whose
$k$-th row is given by the $n$-tuple of coordinates of the
vector $t^{k-1}(v)\in\C^n=V=\Gamma(C, \CF),\,$ $k=1,\ldots,n$.

\begin{lemma}\label{torsor} \vi The  map $\supp: X_n^{\oper{cyc}}\to
C^{(n)}$, cf. \eqref{supp},
makes the  scheme $X_n^{\oper{cyc}}$ a
$GL(V)$-torsor over $C^{(n)}$.

\vii Given  a basis in $V$ and a {\em global} coordinate
$t: C\into\BA^1,$ the map
\begin{equation}\label{triv}{\mathsf{g}}\times \supp:\
X_n^{\oper{cyc}}\iso GL_n \times C^{(n)},\quad
(\CF,\vartheta,v)\mto \big({\mathsf{g}}(\CF,\vartheta,v),\;\Supp\CF\big),
\end{equation}
provides
a $GL(V)$-equivariant trivialization of the $GL(V)$-torsor from 
$\mathsf{(i)}.$
\end{lemma}
\begin{proof} 
The action of $GL(V)$ on $\Quot^n_C\times V$ is given by 
$g(\CF,\vartheta,v)=(\CF,\vartheta\circ g^{-1},gv)$. 

To prove (i), we observe that the quotient  
$GL(V)\backslash X_n^{\oper{cyc}}$ is the moduli stack of quotient sheaves
of length $n$ of the structure sheaf $\CO_C$. However, this stack is
just the Grothendieck Quot scheme $\oper{Quot}^n_{\CO_C}$
isomorphic to $C^{(n)}$. Part (i) follows. Part (ii) is immediate.
\end{proof}

On $X_n=\Quot^n_C\times  V,$ we have a trivial line bundle $\det$
with fiber $\wedge^n V$. This  line bundle comes  equipped with the
natural $GL(V)$-equivariant structure given,
for any $a\in\wedge^n V$,  by
$g(\CF,\vartheta,v;a)=(\CF,\vartheta\circ
g^{-1},$ $gv;(\det g)\cdot a)$. We define a
 {\em determinant 
  bundle} to be the unique line bundle $\BCL$  on $C^{(n)}$
such that  $\supp^*\BCL$,
the pull-back of $\BCL$
via the  projection $X_n^{\oper{cyc}}\onto C^{(n)},$ is isomorphic
to $\det|_{X_n^{\oper{cyc}}}.$ 

The square  $\BCL^2$, of determinant bundle,
 has a canonical
{\em rational} section $\bdel$, defined
as follows. Let $(c_1,\ldots,c_n)$ be
pairwise distinct points of $C$,
and  $\CF=\CO_C/\CO_C(-c_1-\ldots-c_n)$
$=\CO_{c_1}\oplus\ldots\oplus\CO_{c_n},$ a point in $X_n\cyc$. The fiber
of  $\BCL^2$ at the point $\supp(\CF)\in C^{(n)}$  is
identified with the vector space
$\BCL^2_\CF=(\CO_{c_1}\otimes\ldots\otimes\CO_{c_n})^{\otimes2}.$
We define $\bdel(c_1,\ldots,c_n):=
(1_{c_1}\wedge\ldots\wedge1_{c_n})^2$, where
$1_{c_r}\in \CO_{c_r}$ stands for the unit element.

We have the finite  projection map $p: C^n\onto  C^{(n)}$,
the (big)  diagonal divisor,  $\Delta_n\subset C^n$,
and the  discriminant divisor,
 $\bD\subset C^{(n)}$.
Observe that,  set theoretically, we have $|\bD|=p(|\Delta_n|);$
however, for divisors, one has an equation
 $p^*\bD=2\Delta_n$. 

Next, we put $X_n^\reg:=X_n\cyc\cap\supp\inv(C^{(n)}\sminus\bD)$.
Clearly, we have $X_n^\reg\sset X_n\cyc\sset X_n.$
Given a global
coordinate $t:\ C\into \BA^1,$ we write
$$
\pi(c):=\prod_{1\leq i < j\leq n}
\big(t(c_i)-t(c_j)\big), \qquad \forall c=(c_1,\ldots, c_n)\in C^n.
$$ 
for the Vandermonde determinant. Thus, $\pi^2(c)$ is a regular function
on $C^{(n)}$.

\begin{proposition}\label{bdel} \vi For any smooth connected curve $C$,
we have 
$$X_n\sminus X_n^\reg= \supp\inv(\bD)\cup (X_n\sminus X_n\cyc),
$$
 is a union of two irreducible divisors.

\vii There exists a regular function 
$\ff\in \CO(X_n^\reg)$ such that  
\vskip 2pt

\pb{it has a zero of order 2 at the divisor $X_n\sminus X_n\cyc,$
and a pole of order~1 at the divisor $\supp\inv(\bD)$;}

\pb{it is a $GL_n$-semi-invariant, specifically, we have}
$$\ff(g\cdot x)=(\det g)^2\cdot \ff(x),\qquad\forall g\in GL_n, \,x\in
X_n^\reg.
$$

\viii The function $\ff$ is defined
uniquely up to multiplication by the pull back of an invertible function
on $C^{(n)}$.
Furthermore,
in a trivialization as in Lemma \ref{torsor}(ii),
one can put\newline
\centerline{
$\dis\ff(x):=\big(\det {\mathsf{g}}(x)\cdot \pi(\supp(x)\big)^2,\quad x\in X_n^\reg.$}
\end{proposition}
\begin{proof} 
The uniqueness statement follows from the semi-invariant property of
$\ff$. To prove the existence, recall the definition of the
line $\BCL$. By that definition, a section of $\BCL$ is the
same thing as a  semi-invariant section of the trivial bundle with fiber $\wedge^n V$,
on $X_n\cyc$. We let $\ff$ be the rational section of the latter bundle
corresponding, this way, to the rational section $\bdel$ of $\BCL$.
A choice of base in $V$ gives a basis in $\wedge^n V$,
hence allows to view  $\ff$ as a  semi-invariant rational function.

The order of zero of $\ff$ at $X_n\sminus X_n\cyc$ is a local
question, so we may assume the existence of a local coordinate $t$ on
$C$, trivializing our torsor as in~\ref{torsor}(ii). Then the divisor
$X_n\sminus X_n\cyc$ is given by an equation $\det {\mathsf g}(x)=0$.
So we see that the zero of $\ff$ at $X_n\sminus X_n\cyc$ indeed has
order~2. The statement that $\ff$ has a pole of order~1 at
$\supp\inv(\bD)$ immediately follows from the next lemma.
In the presence of a local coordinate $t$,
the explicit formula for the function $\mathsf{g}$ then follows
from the uniqueness statement.
\end{proof}

\begin{lemma}
\label{koren'}
There is a canonical isomorphism $\BCL^{\otimes2}\simeq\CO_{C^{(n)}}(\bD)$.
\end{lemma}

\proof 
We have to construct an $\syn$-equivariant section of $p^*\BCL^{\otimes2}$
which is regular nonvanishing on $C^n-\Delta_n$, and has a second order
pole at (each component of) $\Delta_n$. 
Clearly, $\bdel$ is a $\syn$-invariant regular nonvanishing section
of $p^*\BCL^{\otimes2}|_{C^n\sminus\Delta_n}$. To compute the order of the pole
of $\bdel$ at the diagonal $c_i=c_j$ it suffices to consider
the case where $C=\BA^1,\ n=2$. 
In the latter case, a straightforward calculation shows
that  $\bdel$ has a second order pole.
\endproof

\subsection{Quantum Hamiltonian reduction.}
\label{cvetochki} The goal of this section is to provide a construction
of the sheaf of spherical Cherednik subalgebras
in terms of quantum Hamiltonian reduction, in the spirit of \cite{GG}.

Let $c_1(\BCL)\in H^2(C^n,\Omega_{C^{(n)}}^{1,2})^{\syn}$
denote the image of the first 
Chern class of the determinant line bundle $\BCL$.
 For a complex number $\kappa$,
we will sometimes denote the class $\kappa\cdot c_1(\BCL)\in
H^2(C^n,\Omega_{C^n}^{1,2})^{\syn}$ simply by $\kappa$, if there is
no risk of confusion. 
Let $\psi\in H^2(C,\Omega_C^{1,2})^{\syn}$. Note that if $C$ is
affine, then $\psi=0$, and if $C$ is projective, then $\psi$ is
proportional to the first Chern class $c_1(\ssl)$ of a degree one line bundle
$\ssl$ on $C$, so that $\psi=k\cdot c_1(\ssl)$. 
The line bundle $\ssl^{\boxtimes n}$ on $C^n$ is
$\syn$-equivariant.
Let  $\ssl^{(n)}$ denote the subsheaf of $\syn$-invariants in
$p_*\ssl^{\boxtimes n}$, the direct image sheaf on $C^{(n)}$.

We set $\psi_n=k\cdot c_1(\ssl^{(n)}\in
H^2(C^{(n)},\Omega_{C^{(n)}}^{1,2})$, and
let $\supp^*(\psi_n)\in H^2(X_n,\Omega_{X_n}^{1,2})$, be its pull-back
 via the support-morphism \eqref{supp}, cf. \cite{Eti}, 2.9.
Let   $\D_{\psi_n}(X_n)$ denote the sheaf 
 on $X_n=\rep^n_C\times  V$,   of twisted differential operators
associated with  the class
 $\supp^*(\psi_n)$. 

We have a  $GL_n$-action along the
fibers of  the support-morphism. Therefore,
the TDO   associated with a pull-back class comes equipped
with a natural $GL_n$-equivariant structure.
More precisely, \cite{BB2}, Lemma 1.8.7, implies that
the pair $(\D_{\psi_n}(X_n), GL_n)$ has the canonical structure of
a \hhh algebra on $X_n$, in the sense of \cite{BB2}, \S 1.8.3.
It follows, in particular, that there is a natural
Lie algebra morphism $\gl_n\to\D_{\psi_n}(X_n), \,u\mto\arr{u}$,
compatible with the $GL_n$-action on $X_n$. 

For any $\kappa\in\C,$ 
the assignment $u\mto \arr{u}-\kappa\cdot tr(u)\cdot1$ gives another
Lie  algebra morphism
 $\gl_n\to\D_{\psi_n}(X_n)$.
Let ${\mathfrak g}_\kappa\sset \D_{\psi_n}(X_n)$ denote the
image of the latter morphism.
Thus,  ${\mathfrak g}_\kappa$ is
a Lie subalgebra of first order twisted differential operators,
and we write $\D_{\psi_n}(X_n){\mathfrak g}_\kappa$  for
the left ideal in $\D_{\psi_n}(X_n)$ generated by
the vector space ${\mathfrak g}_\kappa$.

Let $\D_{\psi_n}(X_n^{\oper{cyc}})$ be the restriction of
the TDO $\D_{\psi_n}(X_n)$ to the open subset $X_n\cyc\sset X_n$.
Let $\supp_\idot\D_{\psi_n}(X_n^{\oper{cyc}})$ denote the sheaf-theoretic
direct image, a sheaf of filtered associative algebras
on $C^{(n)}$ equipped with a $GL_n$-action.
 We  have the left ideal
 $\D_{\psi_n}(X_n^{\oper{cyc}}){\mathfrak g}_\kappa$,
in $\D_{\psi_n}(X_n^{\oper{cyc}})$,
and the corresponding $GL_n$-stable left ideal 
$\supp_\idot\D_{\psi_n}(X_n^{\oper{cyc}}){\mathfrak
g}_\kappa\sset\supp_\idot\D_{\psi_n}(X_n^{\oper{cyc}})$. 
\vskip 2pt

The TDO  $\D_\psi(X_n)$ may be identified with a subsheaf of the direct
image of $\D_\psi(X_n^{\oper{cyc}})$ under the open embedding
$X_n^{\oper{cyc}}\hookrightarrow X_n$. This way, one obtains a
restriction morphism
$$
r:\ \supp_\idot\D_\psi(X_n)/\supp_\idot\D_\psi(X_n){\mathfrak g}_\kappa
\to
\supp_\idot\D_\psi( 
X_n^{\oper{cyc}})/\supp_\idot\D_\psi(X_n^{\oper{cyc}}){\mathfrak
g}_\kappa
$$

We are now going to generalize 
\cite{GG} formula (6.15), and construct, for any 
$\kappa\in\C,\,\psi\in H^2(C^n,\Omega_C^{1,2})$,
a canonical  `radial part' isomorphism 
\begin{equation}\label{radmap}{\mathsf{rad}}:\ \big(\supp_\idot\D_{\psi_n}( 
X_n^\reg)/\supp_\idot\D_{\psi_n}(X_n^\reg)
{\mathfrak
g}_\kappa\big)^{GL_n}\iso\D_{\psi_n}(C^{(n)}\sminus\bD),
\quad u\mto\rad{u}_\kappa,
\end{equation}
of sheaves    of filtered associative algebras on
$X_n^\reg/GL_n=C^{(n)}\sminus\bD.$

To this end, assume first that the class $\psi=c_1(\ssl),$ is the first Chern
class of a line bundle $\ssl$ on $C$. Then, we have
$\psi_n=c_1(\ssl^{(n)})$.
Further, we have the pull-back $\supp^*\ssl{}^{(n)}$,  a 
line bundle on $X_n\cyc$ equipped with a natural
$GL_n$-equivariant structure.
Write $\D(C^{(n)}, \ssl{}^{(n)}),$ resp.
 $\D(X_n\cyc,\supp^*\ssl{}^{(n)})$ for the sheaf of TDO
on $C^{(n)},$ acting in the line bundle $\ssl{}^{(n)},$ resp. TDO
on  $X_n^\reg,$ acting in the line bundle $\supp^*\ssl{}^{(n)}.$

It is clear that,  on $X_n^\reg,$ one has a sheaf isomorphism
$$
\ssl{}^{(n)}\iso \wt\ssl{}^{(n)}\!,\en
s\mto \ff\cd \supp^*(s),\en\text{where}\en
\wt\ssl{}^{(n)}:=\{\wt{s}\in \supp_*\supp^*\ssl{}^{(n)}\;\,\big|\;\, g(\wt{s})=
(\det g)\cd \wt{s},\;\forall g\in GL_n\}.
$$

Observe further that, for  any $GL_n$-invariant twisted differential operator
$u\in \D(X_n^\reg,\supp^*\ssl{}^{(n)})$, we have
$u(\wt\ssl{}^{(n)})\sset \wt\ssl{}^{(n)}.$
We deduce that, for any  $\kappa\in\C,$
and any $GL_n$-invariant twisted differential operator
$u\in \D(X_n^\reg,\supp^*\ssl{}^{(n)})$ the assignment
$$\rad{u}_\kappa:\ \ssl{}^{(n)}\to\ssl{}^{(n)},\quad
s\mto \ff^{-\kappa}\cd u(\ff^\kappa\cd s),
$$
is well-defined and, moreover, it is given by a uniquely determined
twisted  differential operator
$\rad{u}_\kappa\in$ $\D_{\psi_n}(C^{(n)}\sminus\bD).$

The above construction of the map $u\mto \rad{u}_\kappa$ may be adapted to cover the general case,
where the class $\psi$ is not necessarily of the form $c_1(\scr L)$
for a line bundle $\scr L$.
The map $u\mto \rad{u}_\kappa$ is, by
definition, the radial part homomorphism $\mathsf{rad}$ that
appears in \eqref{radmap}.

The group $\syn$ acts freely on $C^n\sminus\Delta_n$, and 
we have $(C^n\sminus\Delta_n)/\syn=C^{(n)}\sminus\bD$.
Hence,
$\D_{\psi_n}(C^{(n)}\sminus\bD)\simeq
\D_{\psi^{\boxtimes n}}(C^n\sminus\Delta_n)^\syn$. The
lift of the determinant line bundle $\BCL$ to $C^n$ has a canonical
rational section $\bdel^{\frac{1}{2}}$. The restriction of this section to
$C^n\sminus\Delta_n$ is invertible, and
we let $\mathsf{twist}: \D_{\psi_n}(C^{(n)}\sminus\bD)\iso
\D_{\psi_n+1}(C^{(n)}\sminus\bD)$ be an isomorphism of TDO
induced via
conjugation by $\bdel^{\frac{1}{2}}$.

Combining all the above, we obtain the following morphism
of sheaves of filtered algebras on $C^{(n)}$, where $\mathsf{j}$ stands for the
open embedding $C^{(n)}\sminus\bD\hookrightarrow C^{(n)}$:
\begin{equation}\label{CH}{\mathsf{twist}\ccirc\mathsf{rad}\ccirc r}:\
\left(\frac{\supp_\idot\D_\psi(X_n)}
{\supp_\idot\D_\psi(X_n){\mathfrak g}_\kappa}\right)^{GL_n}
\too
\mathsf{j}_*\D_{\psi_n+1}(C^{(n)}\sminus\bD).
\end{equation}
Here, the algebra on the left
is   the quantum
Hamiltonian reduction of $\D_\psi(X_n^\reg)$ at the point 
 $\kappa\cdot tr\in\gl_n^*$.

Our main result about quantum Hamiltonian reduction reads

\begin{theorem}
\label{sheaf eg} The image of the composite morphism in \eqref{CH}
is equal to  the spherical
Cherednik subalgebra. Moreover, this composite yields
a filtered algebra isomorphism
$$
\big(\supp_\idot\D_{\psi_n}(X_n)/
\supp_\idot\D_{\psi_n}(X_n){\mathfrak g}_\kappa\big)^{GL_n}
\iso \e\sH_{\kappa,\psi_n+1}\e,
$$
as well as the associated graded
 algebra isomorphism
$$
\gr\big(\supp_\idot\D_{\psi_n}(X_n)
/\supp_\idot\D_{\psi_n}(X_n){\mathfrak g}_\kappa\big)^{GL_n}
\iso \gr(\e\sH_{\kappa,\psi_n+1}\e).
$$
\end{theorem}

\subsection{Proof of Theorem \ref{sheaf eg}.} It will be convenient
to introduce the following
simplified notation.
Let
$\CH_{n,\kappa,\psi}\subset\mathsf{j}_*\D_{\psi_n+1}(C^{(n)}\sminus\bD)$ 
be the image
of the composite morphism in~\eqref{CH} (note the shift
$\psi_n+1$ on the right); also,
for the  spherical subalgebra,  write
\begin{equation}\label{CHCA}
\sA_{n,\kappa,\psi}:=\e\sH_{\kappa,\psi_n+1}\e.
\end{equation}

We will use the factorization isomorphism
\eqref{fac} for $k+m=n$. Write $j:{\gen}C^{(k,m)}\into$
$C^{(k,m)}:=C^{(k)}\times C^{(m)}$
for the open imbedding and
$pr: {\gen}C^{(k,m)}\into C^{(k,m)}\onto
C^{(n)}$
for the composite projection.

A natural factorization property for the determinant bundle
gives rise to a canonical isomorphism of sheaves of TDO,
$$pr^*\D_{\psi_n+1}(C^{(n)})\simeq
j^*(\D_{\psi_k+1}(C^{(k)})\boxtimes
\D_{\psi_m+1}(C^{(m)})).$$

\begin{lemma}
\label{factorization}
The above isomorphism restricts to the isomorphism of subalgebras
$pr^*\CH_{n,\kappa,\psi}\simeq
j^*(\CH_{k,\kappa,\psi}\boxtimes\CH_{m,\kappa,\psi})$.
\end{lemma}

\proof Let $(\CF_1,v_1)\in X_k$ and $(\CF_2,v_2)\in X_m$ be such that
the sheaves $\CF_1$ and $\CF_2$ have  disjoint
supports. Then, the stabilizer in $GL_n$ of the
point $(\CF_1\oplus\CF_2,v_1\oplus
v_2)\in X_n,\, n=k+m,$ is equal
to the product of the stabilizers
of $(\CF_1,v_1)$ and $(\CF_2,v_2)$ in $GL_k$ and $GL_m$, respectively.
This yields the following factorization isomorphism,
cf.~\eqref{fac},
\begin{align*}
\left(GL_{k+m}\overset{GL_k\times GL_m}{\times}(X_k\times X_m)
\right)\underset{{C^{(k)}\times C^{(m)}}}\times \ {\gen}C^{(k,m)}
\simeq X_n\times_{C^{(n)}}\ {\gen}C^{(k,m)},\\
(g,\CF_1,\vartheta_1,v_1,\CF_2,\vartheta_2,v_2)\longmapsto
(\CF_1\oplus\CF_2,(\vartheta_1\oplus\vartheta_2)\circ
g^{-1},g(v_1\oplus v_2)).
\end{align*}

Now, the desired isomorphism of subalgebras
in the statement of the lemma is a particular case of the
following situation. We have a subgroup $G'\subset G''$ (in our case
$G'=GL_k\times GL_m,\ G''=GL_n$), and a $G'$-variety $X'$ with a TDO
$\D'$ equipped with an action of $G'$ (in our case $X'=(X_k\times X_m)
\underset{{C^{(k)}\times C^{(m)}}}\times \ {\gen}C^{(k,m)}$, and 
$\D'=j^*(\D_{\psi_k}\boxtimes\D_{\psi_m})$). We have a $G''$-invariant
linear functional $\chi''$ on the Lie algebra $\mathfrak{g}''$ whose
restriction to $\mathfrak{g}'$ is denoted by $\chi'$ (in our case
$\chi''=(\kappa-1)\cdot tr$). We set $X''=G''\overset{G'}{\times}X'$;
it is equipped with the TDO $\D''$ lifted from $X'$, acted upon by
$G''$ (in our case $X''=X_n\times_{C^{(n)}}\ {\gen}C^{(k,m)}$, and 
$\D''=pr^*\D_{\psi_n}$). Then, it is easy to verify that one has
 $(\D''/\D''{\mathfrak
  g}''_{\chi''})^{G''}\simeq (\D'/\D'{\mathfrak
  g}'_{\chi'})^{G'}$.
\endproof

\begin{lemma}\label{DD} Let $U\subset C^{(n)}$ be an open subset
and let  $D_1\in\Gamma(U, \CH_{n,\kappa,\psi})
\subset\D_{\psi_n+1}(U)$,
resp. 
$D_2\in\Gamma(U,\sA_{n,\kappa,\psi})$,
 be a pair of
second order twisted
differential operator on $U$ with equal  principal symbols.

Then, the difference $D_1-D_2$ is a zero order differential operator, more precisely, it is
the operator of multiplication by
a regular function on $U$.
\end{lemma}
\proof
It is enough to prove the statement of the lemma in
the formal neighbourhood of each point $\underline{c}=
(k_1c_1+\ldots+k_lc_l)\in U\subset
C^{(n)}$, where $c_1,\ldots,c_l$ are pairwise distinct and $k_1+\ldots+k_l=n$.
By induction in $n$, we are reduced to the diagonal case
$\underline{c}=(nc)$. Then it suffices to take the formal
disk around $c$ for $C$.
In the latter case, our claim follows from the explicit
calculation in \cite{GGS}, \S5,
and the proof of~Proposition~2.18 in~\cite{Eti}.
\endproof

\proof[Proof of Theorem \ref{sheaf eg}] We equip
the sheaf $\CH_{n,\kappa,\psi}$, defined above \eqref{gralg},
with the standard increasing filtration induced by
 the standard increasing filtration on $\D_{\psi_n}(X_n),$
by the order of differential operator.
Let 
$\gr\CH_{n,\kappa,\psi}$ denote
the  associated graded sheaf. We also have
 the increasing filtration on the algebra $\sA_{n,\kappa,\psi},$
 cf. \eqref{gralg}. Further, Lemma \ref{DD} yields
$F_2\sA_{n,\kappa,\psi}=
F_2\CH_{n,\kappa,\psi}$. 

Next, we observe that the spherical Cherednik subalgebra
$\sA_{n,\kappa,\psi}$ is generated by the subsheaf
$F_2\sA_{n,\kappa,\psi}$. Indeed, it is enough
to prove the equality of $\sA_{n,\kappa,\psi}$ with the
 subsheaf  generated by $F_2\sA_{n,\kappa,\psi}$ 
in the formal neighbourhood of any point $\underline{c}\in
C^{(n)}$. Arguing by induction in $n$ as in the proof of the lemma, it
suffices to consider the case $\underline{c}=(nc)$.
We then to take the formal
disk around $c$ for $C$. In such a case, the desired statement is proved 
in~Section~10 of~\cite{EG}.

Thus, we have  filtered algebra morphisms 
$$\sA_{n,\kappa,\psi}\hookrightarrow
\CH_{n,\kappa,\psi}\twoheadleftarrow
(\supp_\idot\D_{\psi_n}(X_n)/\supp_\idot\D_{\psi_n}(X_n){\mathfrak
  g}_\kappa)^{GL_n}.$$ 
To prove that these morphisms are isomorphisms, it suffices 
to show that the induced morphisms of associated graded algebras,
$$\gr\sA_{n,\kappa,\psi}\to
\gr\CH_{n,\kappa,\psi}\leftarrow
\gr(\supp_\idot\D_{\psi_n}(X_n)/\supp_\idot\D_{\psi_n}(X_n){\mathfrak
  g}_\kappa)^{GL_n},$$ 
are isomorphisms. Reasoning as above, we may further reduce the
proof of this last statement to the
case where $C$ is the formal disk around $c$. The latter case follows from
~\cite{GG}, page~40. The theorem is proved.
\endproof

\section{Character sheaves}\label{char_sec}
\subsection{The moment map.} Fix   a
smooth curve $C$ and let $\De\into C^2=C\times C$ be the diagonal.

Recall  the smooth scheme $X_n=\rep_C^n\times V$, and let
$T^*X_n=(T^*\rrep)\times  V\times V^*$ denote the total
space of the cotangent bundle on $X_n$. We will write a point of $T^*X_n$
as a quadruple 
$$(\CF, y, i,j)\quad\text{where}\
\CF\in \rrep,\;y\in \Ga\big(C^2,
(\CF\boxtimes\CF^\vee)(-\De)\big),\,i\in V,\;j\in V^*.
$$

The group $GL(V)$ acts diagonally on
 $\rrep\times V$. This gives  a Hamiltonian $GL(V)$-action on $T^*X_n,$
with moment map $\mu$.  Given $\CF\in \rrep,$ we will also use the map
\begin{equation}\label{nu}
\nu_\CF: \
\Ga\big(C^2,
(\CF\boxtimes\CF^\vee)(-\De)\big)\into
\Ga(C^2,
\CF\boxtimes\CF^\vee)\iso \End V,
\end{equation}
induced by
the sheaf imbedding $\oo_{C^2}(-\De)\into \oo_{C^2}$, cf. \eqref{can}.

\begin{lemma} The moment map $\mu$
is given by the formula, cf. \eqref{nu},
$$\mu:\
(T^*\rrep)\times  V\times V^*\to \End V=\Lie GL(V),
\quad (\CF,y,i,j)\mto \nu_\CF(y)+ i\o j.
$$
\end{lemma}

We leave the proof to the reader.

\begin{example} 
\label{primer}
In the special  case $C=\BA^1$, we have
$\rrep=\End V\cong\mathfrak{gl}_n$. In this case, the moment map reads,
see
\cite{GG},
$$\mu:\ \mathfrak{gl}_n\times \mathfrak{gl}_n\times  V\times V^*\too\mathfrak{gl}_n,\quad
(x,y,i,j)\mto [x,y] +i\o j.
$$

Similarly, in the case $C=\C^\times$,  the moment map reads
$$
\mu:\ GL_n\times\mathfrak{gl}_n\times V\times V^*\too\mathfrak{gl}_n,
\quad(x,y,i,j)\mto
xyx^{-1}-y+i\o j.$$
\end{example}

\subsection{A categorical quotient.}
The goal of this
subsection is to construct an isomorphism
$(T^*C)^{(n)}\simeq\mu^{-1}(0)/\!/GL_n$ (the categorical quotient).
To stress the dependence on $n$ we will sometimes write $\mu_n$ for the
moment map $\mu$.

Note that we have the direct sum morphism 
$$X_k\times X_m\to X_{k+m},\
(\CF,\vartheta_1,v_1;\CG,\vartheta_2,v_2)\mapsto(\CF\oplus\CG,
\vartheta_1\oplus\vartheta_2,v_1\oplus v_2).
$$
This morphism induces a similar   direct
sum morphism $T^*X_k\times T^*X_m\to T^*X_{k+m}$, 
{\small
$$\Gamma(\CF\boxtimes\CF^\vee(-\Delta))\times
\Gamma(\CG\boxtimes\CG^\vee(-\Delta))\ni (y_1,y_2)\mto
y_1\oplus y_2\in \Gamma((\CF\oplus\CG)\boxtimes
(\CF\oplus\CG)^\vee(-\Delta)),
$$
}
where we have used simplified notation $\Gamma(-)=\Gamma(C^2,-)$.

Clearly, we have
$T^*X_1=T^*C\times\BA^1\times(\BA^1)^*$. Iterating the direct sum
morphism $n$ times, we obtain a morphism
$(T^*C\times\BA^1\times(\BA^1)^*)^n\to T^*X_n$. Restricting this map
further to the
 product of $n$ copies of the
subset $T^*C\simeq T^*C\times\{1\}\times\{0\}\subset
T^*C\times \BA^1\times(\BA^1)^*$, we obtain a morphism 
$(T^*C)^n\to T^*X_n$. The image 
of the latter morphism  is
clearly contained in $\mu_n^{-1}(0)$. It is also clear that the
composite morphism $(T^*C)^n\to T^*X_n\to T^*X_n/\!/GL_n$ factors
through the symmetrization projection $(T^*C)^n\to(T^*C)^{(n)}$. 

This way, we have constructed a morphism
$\varepsilon:\
(T^*C)^{(n)}\to\mu_n^{-1}(0)/\!/GL_n$.

\begin{lemma}
\label{evid}
$\varepsilon:\ (T^*C)^{(n)}\to\mu_n^{-1}(0)/\!/GL_n$ is an isomorphism.
\end{lemma}

\begin{proof} Use the factorization and the ``local'' result for $C=\BA^1$
proved in~\cite{GG}~2.8.
\end{proof}

\begin{notation}\label{pi}
We let $\pi_n$ denote the natural projection $\mu_n^{-1}(0)\to
\mu_n^{-1}(0)/\!/GL_n\simeq$ $(T^*C)^{(n)}$. If the value of $n$ is clear from
the context, we will simply write $\pi:\ \mu^{-1}(0)\to(T^*C)^{(n)}$.
\end{notation}

\subsection{Flags.}\label{flags}
Fix a smooth curve $C$ and a sheaf $\CF\in \rrep.$ Let
$
\CF_\idot:\ 0=\CF_0 \sset\CF_1\sset\ldots\sset\CF_n=\CF
$
be a complete flag of subsheaves, $\operatorname{length}(\CF_r)=r$.
Thus, for each $r=1,2,\ldots,n$, we have a sheaf imbedding
$i_r: \CF_r\into \CF,$ and also the dual projection
$i_r^\vee: \CF^\vee\onto\CF_r^\vee.$

Thus, for the spaces of global sections of sheaves on $C^2$, we have a diagram
$$
\xymatrix{
\Ga\big(
(\CF_{r-1}\boxtimes\CF_r^\vee)(-\De)\big)
\ar@{^{(}->}[rr]^<>(0.5){i_{r-1}}&&
\Ga\big(
(\CF\boxtimes\CF_r^\vee)(-\De)\big)&&
\Ga\big(
(\CF\boxtimes\CF^\vee)(-\De)\big)
\ar@{->>}[ll]_<>(0.5){i^\vee_r}
}.
$$

\begin{definition}\label{nilflag} (a) Let $\widetilde{\rep_C^n}$ be the
  moduli scheme of pairs $(\CF,\CF_\idot),$ where $\CF\in\rep_C^n$,
  and $\CF_\idot$ is a flag of subsheaves.

(b) Fix $\CF\in \rrep,$ and an  element $y\in \Ga\big(C^2,
(\CF\boxtimes\CF^\vee)(-\De)\big)$. 
A flag $\CF_\idot$
is  said to be a {\em nil-flag} for the pair $(\CF,y)$ if we have
$$i^\vee_r(y)\in i_r(\Ga\big(C^2,
(\CF_{r-1}\boxtimes\CF_r^\vee)(-\De)\big),
\quad\forall r=1,2,\ldots,n.
$$
\end{definition}

We have a natural forgetful morphism $
\widetilde{\rep_C^n}\to\rep_C^n$,
a $GL(V)$-equivariant proper  morphism which
 is an analogue of Grothendieck simultaneous
resolution, cf.~\cite{La}~3.2.

We further extend the above morphism to a map
$\phi:\
\widetilde{\rep_C^n}\times V\to X_n=\rep_C^n\times V,$
identical on the second factor $V$.
For the corresponding cotangent bundles, one obtains a standard  diagram
$$
\xymatrix{
T^*X_n&&\phi^*(T^*X_n)\ar[ll]_<>(0.5){p_1}
\ar@{^{(}->}[rr]^<>(0.5){p_2}&&
T^*(\widetilde{\rep_C^n}\times V),}
$$
where the map $p_2$ is a natural closed imbedding, and
$p_1$ is a proper morphism. Thus,
 $\phi^*(T^*X_n)$ is the smooth variety parametrizing
quintuples $(\CF,\CF_\idot,y,i,j)$. 

Let $Z\sset T^*(\widetilde{\rep_C^n}\times V)$ be a closed 
subscheme  formed by the quintuples $(\CF,\CF_\idot,y,i,j)$ such
that $\CF_\idot$ is a nil-flag for $y$.
Now, the set
$p_1^{-1}(\mu^{-1}(0))$ is clearly a closed algebraic
subvariety in $\phi^*T^*(\rep_C^n\times V)$, and so is
$p_1^{-1}(\mu^{-1}(0))\cap p_2^{-1}(Z)$.
The map $p_1:\ p_1^{-1}(\mu^{-1}(0))\cap p_2^{-1}(Z)\to T^*X_n$
is a proper projection.

Observe next that the zero section embedding $C\hookrightarrow T^*C$
induces
an embedding
$C^{(n)}\hookrightarrow(T^*C)^{(n)}$, and recall the map $\pi$
from Notation \ref{pi}.

\begin{proposition}\label{nil2} For a quadruple 
$(\CF,y,i,j)\in \mu\inv(0)$ the following are equivalent

\pb{The pair $(\CF,y)$ has a nil-flag;}

\pb{$(\CF,y,i,j)$ belongs to the image of the projection
$p_1:\ p_1^{-1}(\mu^{-1}(0))\cap p_2^{-1}(Z)\to T^*X_n$.}

\pb{$\pi(\CF,y,i,j)\in C^{(n)}\subset(T^*C)^{(n)}$.}
\end{proposition}

This proposition will be proved in the next section.

\subsection{A Lagrangian subvariety.}\label{lagrangian}
We define a {\em nil-cone} $\Nnil(C)\sset T^*X_n$ to be the set of
points satisfying
the equivalent conditions of Proposition \ref{nil2}, with the natural
structure of a reduced closed subscheme of $T^*X_n$ arising from the
third condition of the proposition.

 In the special case of the curve
$C=\BA^1,$ we have
$\rrep=\End V\cong\mathfrak{gl}_n,$ and
the corresponding variety $\Nnil(C)$ is nothing but
the Lagrangian  subvariety  introduced in~\cite{GG}, (1.3).
In more detail, write  $\CN\sset \mathfrak{gl}_n$ for the nilpotent variety.
Then, one has 

\begin{lemma}\label{line} \vi Let
$C=\BA^1$. Then,  we have
\begin{equation}\label{mnil}
\Nnil(\BA^1)=
\{(x,y,i,j)\in \mathfrak{gl}_n\times \mathfrak{gl}_n\times  V\times V^*
\en\big|\en  [x,y] +i\o j=0\en\;\&\;\en 
y\in\CN\}.
\end{equation}

\vii  Similarly, in the case $C=\C^\times$, we have
$$
\Nnil(\C^\times)\simeq\{(x,y,i,j)\in
GL_n\times\mathfrak{gl}_n\times V\times V^*
\mid
xyx^{-1}-y+i\o j=0\en\;\&\;\en  y\in\CN\}.
$$
\end{lemma}

\begin{proof} In the case $C=\BA^1$, 
the morphism $\phi:\
\widetilde{\rep_{\BA^1}^n}\to\rep_{\BA^1}^n=\mathfrak{gl}_n$ 
becomes the
Grothendieck simultaneous resolution, cf.~\cite{La}~(3.2). 
In this case  Proposition \ref{nil2} is known. 
Hence we only have to check that for any quadruple $(x,y,i,j)$ as 
in~\eqref{mnil} there exists a complete flag $V_\idot$ such that $xV_k\subset
V_k$, and $yV_k\subset V_{k-1}$ for any $k=1,\ldots,n$. But this is
also well known, see eg. \cite{EG}, Lemma 12.7. Part (i) follows. 
Part (ii) is proved similarly.
\end{proof}
\begin{proof}[Proof  of Proposition \ref{nil2}]
The first two conditions are clearly equivalent.
The first and third conditions are compatible with factorization, 
and so are reduced
to the ``local'' case $C=\BA^1$, see~Example~\ref{primer}. 
Then they are both equivalent to the nilpotency of $y$,
cf.~Lemma~\ref{line}. 
\end{proof}

\begin{proposition}\label{nil_prop} $\Nnil(C)$ is the reduced scheme
  of a Lagrangian locally complete intersection subscheme in 
$T^*(\rep_C^n\times V)$.
\end{proposition}

We recall that it
 has been  proved in~\cite{GG},~Theorem~1.2, that the right hand side in
\eqref{mnil} is the reduced scheme of a
lagrangian  complete intersection in
$\mathfrak{gl}_n\times \mathfrak{gl}_n\times  V\times V^*$.

\begin{proof}[Proof of Proposition \ref{nil_prop}]
We are going to reduce the statement of the
proposition to the  above mentioned special case of $C=\BA^1$,
see \eqref{mnil}.

To this end, fix a curve $C'$.
Proposition \ref{etale_prop} implies that
there exists a finite collection of Zariski open subsets $C\sset C'$
and, for each subset $C$, an
\'etale morphism $p:\ C\to\BA^1$ 
and a
Zariski open subset
$\CU\sset \rep^n_C\times V$, such that
the following holds:

\pb{The scheme $\rep_{C'}^n\times V$ is covered by the images 
in $\rep^n_{C'}$ of the
 subsets $\CU$, corresponding to the curves $C$ from our collection;}

\pb{For each $C$, there is an \'etale morphism of curves
  $ p:\ C\to\BA^1$
such that the induced morphism below is also \'etale,}
$$\rep p\times \Id_V:\
\rep_C^n\times V\too X:=\rep_{\BA^1}^n\times V={\mathfrak{gl}}(V)\times V.
$$

Then,  we have $T^*\CU=\CU\times_X T^*X$, hence, the map $\rep p\times \Id_V$
induces a
$GL(V)$-equivariant \'etale morphism $ p_*:\ T^*\CU\to T^*X$. The moment
map $T^*\CU\to\End V$ factors through $T^*\CU\to T^*X\to\End V$, and 
$\Nnil(C)\cap T^*\CU$ is the preimage of $\Nnil(\BA^1)$ in $T^*\CU$.
Thus, $\Nnil(C)\cap T^*\CU$ is the reduced scheme of a lagrangian locally 
complete intersection  in $T^*\CU$.
The proposition follows.
\end{proof}

\subsection{Projectivization.} One can perform the quantum hamiltonian
reduction of Theorem~\ref{sheaf eg} in two steps: first with respect
to the central subgroup  $\BC^\times\subset GL_n$, of scalar matrices, and
then with respect to the subgroup $SL_n\subset GL_n$. 
The subgroup
 $\BC^\times\subset GL_n$ acts trivially on $\rep_C^n$; it 
also acts naturally
on $V=\C^n,$ by dilations. We
put 
$$\Vo:=V\sminus\{0\};\quad \P=\P(V)=\Vo/\BC^\times;\quad \X_n=\rep_C^n\times\P.
$$

We have a natural  Hamiltonian  $\BC^\times$-action on $T^*V$,
a  symplectic manifold, and
Hamiltonian reduction procedure replaces $T^*V$ by $T^*\P$.

Similarly, 
the diagonal  $\BC^\times$-action on $\rep_C^n\times\Vo$
gives rise to a Hamiltonian  $\BC^\times$-action on
$T^*(\rep_C^n\times\Vo),$
with moment map $\mu$. It is clear that the space
$$
T^*\X_n=\mu\inv(0)/\C^\times=\{(x,y,i,j)\in SL_n\times\sv_n\times
\Vo\times V^*
\en|\en \langle j,i\rangle=0\}/\C^\times,
$$
is a  Hamiltonian  reduction of the symplectic manifold
$T^*(\rep_C^n\times\Vo).$

One may also consider
the Hamiltonian reduction of the  Lagrangian subscheme $\Nnil\sset T^*X_n$.
This way one gets a closed Lagrangian subscheme
 ${\mathfrak M}_{\textsf{nil}}(C):=$ $\dis\big(\Nnil(C)\cap\mu\inv(0)\big)\big/\C^\times\sset T^*\X_n$.

Recall that, by Hodge theory, there are natural isomorphisms
$H^2(\P, \Om^{1,2}_\P)\cong H^{1,1}(\P,\C)=\C$.
Therefore, for any
 $\psi\in H^2(C, \Om^{1,2}_C)$ and $c\in\C=H^2(\P, \Om^{1,2}_\P)$,
there is a well defined class $(\psi_n,c)\in
H^2(\X_n, \Om^{1,2}_{\X_n})$,
and the corresponding TDO $\D_{\psi,c}(\X_n)$,  on $\X_n$.

We have a natural algebra isomorphism
\begin{equation}\label{c-red}\D_{\psi,c}(\X_n)\cong
\big(\D_{\psi_n}(X_n)/\D_{\psi_n}(X_n)\cd({\mathsf{eu}}-c)\big)^{\BC^\times},
\end{equation}
where ${\mathsf{eu}}$ is the Euler vector
field
on $V$ that corresponds to the action of the identity matrix
$\Id\in\gl(V).$
The algebra on the right hand side  of \eqref{c-red} is a quantum hamiltonian reduction
 with respect to the group $\C^\times,$ 
at the point $\kappa=c/n$.

\begin{definition}\label{characterDmod}
An $SL(V)$-equivariant twisted $\D_{\psi,c}(\X_n)$-module is called
a {\em character $\D$-module} if its characteristic
variety is contained  in  $\dis
{\mathfrak M}_{\textsf{nil}}(C)=\big(\Nnil(C)\cap\mu\inv(0)\big)\big/\C^\times$.

Let ${\scr C}_{\psi,c}$ denote the (abelian) category of
character  $\D_{\psi,c}(\X_n)$-modules.
\end{definition}

It is clear, by Proposition \ref{nil2},
that any object of the category
${\scr C}_{\psi,c}$ is a holonomic $\D$-module; in particular, such an object
has finite length. 

\begin{remark}
For a general curve $C$ and a general pair
$(\psi,c)\in H^2(C, \Om^{1,2}_C)\times \C,$
the category ${\scr C}_{\psi,c}$ may have no
nonzero objects at all.
It is an interesting open problem, to
analyze which nonzero objects
of the  category ${\scr C}_{0,0}$ admit (at least  formal)
deformation in the direction of some $(\psi,c)\neq(0,0)$.

In the special case of the curve $C=\BC^\times$,
however, we have $H^2(C, \Om^{1,2}_C)=0$. Thus,
there is only one nontrivial
 parameter, $c\in\C$. It turns out that, for any $c\in\C$,
 there is a lot of nonzero
character $\D$-modules on $GL_n\times\P$; moreover, 
all these $\D$-modules have regular singularities.
\hfill$\lozenge$
\end{remark}

\subsection{Hamiltonian reduction functor.} 
We now introduce a version of category
$\CO$ for 
$\sA_{\kappa,\psi}=\e\sH_{\kappa,\psi_n+1}\e$,
the spherical Cherednik algebra
associated with a smooth curve $C$, see \eqref{CHCA}.
\begin{definition}\label{catO}  Let  $\CO(\sA_{\kappa,\psi})$
be the full subcategory of the abelian category of left
$\sA_{\kappa,\psi}$-modules whose objects are 
coherent as $\CO_{C^{(n)}}$-modules.
\end{definition}

For any  $SL(V)$-equivariant
$\D_{\psi,c}(\X_n)$-module $\CF$,
one has the  sheaf-theoretic push-forward
$\supp_\idot\CF$,
a sheaf on $C^{(n)}$, cf. \eqref{supp}.
The latter sheaf comes equipped with
 a natural locally finite (rational) $SL(V)$-action,
and we write
$(\supp_\idot\CF)^{SL(V)}\sset \supp_\idot\CF,$ for the $\CO_{C^{(n)}}$-subsheaf of
$SL(V)$-fixed 
sections. 

Thanks to Theorem \ref{sheaf eg}, one can
apply the general formalism of Hamiltonian reduction,
as outlined in \cite{GG}, \S7, to the spherical Cherednik algebra.
Specifically, we have the following result
\begin{proposition}\label{ham_fun}
\vi One has the following {\em exact} functor
of {\em Hamiltonian reduction}:
$$\BH:\ {\scr C}_{\psi,c} \too \CO(\sA_{\kappa,\psi}),
\quad \CF\mto \BH(\CF)=\CF(\supp_\idot\CF)^{SL(V)},\qquad
\kappa=c/n.
$$
Moreover, this functor induces an equivalence
${\scr
C}_{\psi,c}/\ker\BH\iso\CO(\sA_{\kappa,\psi})$.
\vskip 3pt

\vii The functor $\BH$ has a left adjoint functor
$${}^\top\BH:\
\CO(\sA_{\kappa,\psi})\to
{\scr C}_{\psi,c},\quad
M\mto (\D_{\psi,c}(\X_n)/\D_{\psi,c}(\X_n){\mathfrak g}_\kappa)
\otimes_{\sA_{\kappa,\psi}}\,M.
$$
Moreover, for any $M\in \CO(\sA_{\kappa,\psi})$, the canonical adjunction
$\BH\ccirc{}^\top\BH(M)\to M$ is an isomorphism.
\end{proposition}

To prove Proposition \ref{ham_fun}, first recall  the projection
$T^*C^n\onto (T^*C)^{(n)}.$ Abusing the language we will refer
to the image of the zero section of $T^*C^n$ under the projection as
the `zero section' of $(T^*C)^{(n)}.$

Recall that the spherical algebra  $\sA_{\kappa,\psi}$
has an increasing filtration. Therefore,
given a coherent $\sA_{\kappa,\psi}$-module
$M$, one has a well-defined notion of {\em characteristic variety},
$SS(M)\sset\Spec(\gr^F\sA_{\kappa,\psi}).$
We may (and will) use isomorphism \eqref{gralg} and view
$SS(M)$ as a closed algebraic subset of $(T^*C)^{(n)}$.
Then, the following is clear

\begin{lemma}\label{zerosec} A coherent 
$\sA_{\kappa,\psi}$-module $M$
is an object of $\CO(\sA_{\kappa,\psi})$ if and only if 
$SS(M)$ is contained in the zero section of $(T^*C)^{(n)}.$
\qed
\end{lemma}

\begin{proof}[Proof of Proposition  \ref{ham_fun}]
The assignment $\BH: \CF\mto(\supp_\idot\CF)^{SL(V)}$
clearly gives
a functor from the category of  $SL(V)$-equivariant
coherent $\D_{\psi,c}(\X_n)$-modules to the category of  
quasi-coherent $\CO_{C^{(n)}}$-modules. This functor is
{\em exact} since the support-morphism $\supp$ is affine
and $SL(V)$ is a reductive group acting rationally
on $\supp_\idot\CF$.
Thus, our Theorem \ref{sheaf eg} combined with
\cite{GG}, Proposition 7.1, imply that
$(\supp_\idot\CF)^{SL(V)}$ has a natural structure of  coherent
$\sA_{\kappa,\psi}$-module.

We may use  Lemma \ref{evid} and the map $\pi$ introduced
after the lemma, to obtain a diagram
\begin{equation}\label{mdiag}
\Spec\big(\gr \D_{\psi,c}(\X_n)\big)= T^*\X_n
\supset \mu\inv(0)
\stackrel{\pi}\too
(T^*C)^{(n)}=\Spec
\big(\gr\sA_{\kappa,\psi}\big).
\end{equation}

We remark that the statement of  Lemma \ref{evid} involves 
the space $X_n$ rather than $\X_n$; the above
 diagram is obtained from a similar diagram for the subset
 $\rep_C^n\times \Vo\sset X_n,$ by Hamiltonian reduction with respect to
the $\C^\times$-action on $T^*(\rep_C^n\times \Vo)$ induced
by the natural action of the group
$\C^\times\sset GL(V)$ on $\rep_C^n\times \Vo$. Further, from Lemma \ref{evid}
we deduce that the map
$\pi$ in \eqref{mdiag} induces an
isomorphism $\mu\inv(0)/\!/SL(V)\cong(T^*C)^{(n)}$.

Now let $\CF$ be an $SL(V)$-equivariant
coherent $\D_{\psi,c}(\X_n)$-module.
Choose a good increasing filtration on $\CF$
by $SL(V)$-equivariant $\CO_{\X_n}$-coherent subsheaves
and view $\gr\CF$, the associated graded object,
as a coherent sheaf on $T^*\X_n$. It follows
that $SS(\CF)=\Supp(\gr\CF)\sset \mu\inv(0)$.

The functors $\supp_\idot$ and $(-)^{SL(V)},$ each
being exact, we deduce 
$\gr((\supp_\idot\CF)^{SL(V)})\cong(\pi_\idot\gr\CF)^{SL(V)}$.
Moreover, the isomorphism
$(T^*C)^{(n)}\cong\mu\inv(0)/\!/SL(V)$ insures
that $(\pi_\idot\gr\CF)^{SL(V)}$ is a coherent
sheaf on $(T^*C)^{(n)}$. Hence, the filtration
on $\BH(\CF)=(\supp_\idot\CF)^{SL(V)},$ induced by the one on $\CF,$ is
a good filtration, that is,
$\gr\BH(\CF)$ is a coherent
$\gr\sA_{\kappa,\psi}$-module, cf. \eqref{mdiag}. We conclude that
$SS(\BH(\CF))\sset \pi(SS(\CF)).$

The above implies that, for any character $\D$-module $\CF$,
one has 
$$SS(\BH(\CF))\sset\pi(SS(\CF))\sset \pi({\mathfrak M}_{\textsf{nil}}(C))=
\text{zero section of}\en (T^*C)^{(n)}.
$$

The first claim of part (i) of the proposition follows from these
inclusions and Lemma \ref{zerosec}.
Part (ii) is proved similarly, using that,
for any coherent $\sA_{\kappa,\psi}$-module $M$,
one has
\begin{equation}\label{sinv}SS({}^\top\BH(M))\sset \pi\inv(SS(M)).
\end{equation}

At this point, the second 
 claim of part (i) is a general consequence of the
existence of a left adjoint functor,
 cf. \cite{GG}, Proposition 7.6, and we are done.
\end{proof}

\begin{corollary} \label{findimcor}
Let $\CF$ be a simple character $\D$-module such that $\BH(\CF)\neq0.$
 Then, we have
$$\Gamma(C^{(n)}, \BH(\CF))<\infty\quad\Longleftrightarrow\quad
\supp(\Supp\CF)\en{\operatorname{is\en a\en finite\en subset\en of}}\en C^{(n)}.
$$
\end{corollary}
\begin{proof} The space
of global sections of 
a  coherent
$\CO_{C^{(n)}}$-module is  finite dimensional
if and only if  the module has finite support. Further,
we know that $\BH(\CF)$ is a coherent
$\CO_{C^{(n)}}$-module and, moreover,
it is clear that $\Supp\BH(\CF)\sset\supp(\Supp\CF)$.
This gives the implication `$\Leftarrow$'.

To prove the opposite implication, let $M:=\BH(\CF)$.
We have a {\em nonzero} morphism
${}^\top\BH(M)={}^\top\BH\ccirc \BH(\CF)\to\CF$, that
corresponds to $\Id_{\BH(\CF)}$ via the adjunction isomorphism
$$\Hom_{\scr C_{\psi,c}}({}^\top\BH\ccirc \BH(\CF),\CF)
=\Hom_{\CO(\sA_{\kappa,\psi})}(\BH(\CF),\BH(\CF)).
$$

This yields  the following
inclusions
$$\supp(\Supp\CF)\sset\supp(\Supp{}^\top\BH(M))
\sset \supp\big(\supp\inv(\Supp M)\big)=\Supp M=\text{finite set}.
$$
Here, the leftmost inclusion follows since $\CF$ is
simple, hence the map ${}^\top\BH(M)\to\CF$ is surjective,
and  the rightmost equality is due to our assumption that
the sheaf $\BH(\CF)$ has finite support.
The implication `$\Rightarrow$' follows.
\end{proof}

\section{The trigonometric case}

\subsection{Trigonometric Cherednik algebra.}
\label{Trig}
From now on, we consider a
special case of the curve $C=\BC^\times$.
Thus, we have $\rep_{\BC^\times}^n\simeq GL_n$.

Let $H\sset GL_n$ be the maximal torus formed by diagonal matrices, 
and let $\h$ be the Lie algebra of $H$. Let
 $\{\sy_1,\ldots,\sy_n\}$
and $\{x_1,\ldots,x_n\}$ be dual bases  of $\h$ and $\h^*$,
respectively. The coordinate ring 
$\BC[H]$, of  the torus $H$, 
may be identified with
$\BC[\bx_1^{\pm1},\ldots,\bx_n^{\pm1}],$ the Laurent polynomial ring 
in the variables $\bx_i=\exp(x_i)$.
Given an $n$-tuple $\nu=(\nu_1,\ld,\nu_n)\in\Z^n,$
we write $\bx^\nu:=\bx_1^{\nu_1}\cdots\bx_n^{\nu_n}\in
\BC[\bx_1^{\pm1},\ldots,\bx_n^{\pm1}],$
for the corresponding monomial.

Let ${\mathbb T}=SL_n \cap H$ be a maximal torus in $SL_n$,
and let ${\mathfrak t}:=\Lie {\mathbb T}.$
The coordinate ring  $\BC[{\mathbb T}],$ 
of the subtorus ${\mathbb T}\subset H$, is the 
quotient ring of $\BC[H]$ by the relation $\bx_1\cdots\bx_n=1$.
The diagonal Cartan torus $H^0$ of $PGL_n$ is the quotient torus of $H$,
and $\BC[H^0]\subset\BC[H]$ is the subring generated by products
$\bx_i\bx_j^{-1},\ 1\leq i,j\leq n$.

Let $P=\Hom(H,\C^\times)$, resp. $P_0=\Hom({\mathbb T}, \C^\times)$, 
and $P^0=\Hom(H^0, \C^\times)$, be the weight lattice
of the group $GL_n$, resp. of the group $SL_n$, and $PGL_n$.
Thus, $P^0\subset P_0,$ and we put $\Omega:=P_0/P^0\cong
{\mathbb Z}/n{\mathbb Z}$.
We form semi-direct products
$W^e=P\rtimes W$ (an extended affine Weyl group),
resp. $W^a:=P^0\rtimes W$, and $W^e_0:=P_0\rtimes W$.
One has an isomorphism $W^e_0=W^a\rtimes \Omega$.

We fix $\kappa\in\BC$. The {\em trigonometric Cherednik algebra 
$\sH^\trig_\kappa(GL_n)$ of type
$GL_n$} is generated by the subalgebras $\BC[W^e]:=\BC[H]\rtimes\BC[\syn]$ and 
$\Sym(\h)=\BC[\sy_1,\ldots,\sy_n]$ with 
relations, see e.g.~\cite{AST},~1.3.7, or~\cite{Su},~\S2: 
\begin{align*}
&s_i\cdot \sy-s_i(\sy)\cdot s_i=
-\kappa\langle x_i-x_{i+1},\sy\rangle,
&
\forall \sy\in\h,\ 1\leq i<n;\\
&[\sy_i,\bx_j]=\kappa\bx_j\cdot s_{ij}, & 1\leq i\ne j\leq n;\\
&[\sy_k,\bx_k]=\bx_k-\kappa\bx_k\cdot\sum_{i\in[1,n]\sminus\{k\}}s_{ik}, &
1\leq k\leq n.
\end{align*}

Recall that $\P:=\P(V)$, where we put $V=\C^n$.
According to~\cite{GG},~(6.13), one has a canonical isomorphism
$$(\D({GL_n}\times V)/\D({GL_n}\times V){\mathfrak g}_\kappa)^{GL_n}\simeq
(\D_{n\kappa}({GL_n}\times\P)/\D_{n\kappa}({GL_n}\times\P)
{\mathfrak g}_{\kappa})^{SL_n},$$
where $\D_{n\kappa}({GL_n}\times\P)=\D(GL_n)\o\D_{n\kappa}(\P)$ stands for the
sheaf of twisted differential operators on
${GL_n}\times\P,$ with a twist $n\cdot\kappa$ along $\P$.

\begin{corollary}
\label{reduc GL_n}
The spherical trigonometric Cherednik subalgebra
$\e\sH_\kappa^\trig(GL_n)\e$ is isomorphic to the quantum Hamiltonian reduction
$(\D_{n\kappa}({GL_n}\times\P)/\D_{n\kappa}({GL_n}\times\P)
{\mathfrak g}_{\kappa})^{SL_n}$.
\end{corollary}

\begin{proof}
Recall that $\rho:=\frac{1}{2}\sum_{\alpha\in R^+}\alpha\in\h^*$.
For $\sy\in\h$ let $\partial_\sy$  denote the corresponding 
translation invariant vector
field on $H$. 

Introduce
the {\em Dunkl-Cherednik operator} $T_\sy^\kappa$ as an endomorphism of
$\BC[H]$ defined as follows:
\begin{equation}
\label{dc}
T_\sy^\kappa:=\partial_\sy-\kappa\langle\rho,\sy\rangle+
\kappa\sum_{1\leq i<j\leq n}
\frac{\langle x_i-x_j,\sy\rangle}{1-\bx_i^{-1}\bx_j}(1-s_{ij})
\end{equation}


Let $A(GL_n)$ be a subalgebra of $\End(\BC[H])$ generated by
the operators corresponding to the action of elements $w\in\syn$,
by the commutative algebra  $\BC[H]$, of multiplication operators,
 and by all the operators $T_\sy^\kappa,\
\sy\in\h$.

The following key result is due to Cherednik; for a nice exposition
see e.g.~\cite{Op},~3.7:
\vskip 3pt

\noindent
{\em The  assignment}
 $$ \bx^\nu w\mapsto\exp(\nu)w,\quad
\h\ni \sy\mapsto T_\sy^\kappa,
\qquad\nu\in\Z^n,\,w\in\syn,\,\sy\in\h,$$
{\em extends to an algebra
isomorphism $\sH^\trig_\kappa(GL_n)\iso A(GL_n)$.}

 To complete the proof, observe that,
the curve $(\BC^\times)^{(n)}$ being affine, one can replace
the sheaf of Cherednik algebras
by the corresponding  algebra of  global sections.
 Moreover, since
$H^2(\BC^\times,\Omega^{1,2}_{\BC^\times})=0$,
the parameter $\psi$ in $\sH_{\kappa,\psi}$ vanishes.
 Thus, Theorem~\ref{sheaf eg} says
that
the Hamiltonian reduction  algebra
$(\D_{n\kappa}({GL_n}\times\P)/\D_{n\kappa}({GL_n}\times\P)
{\mathfrak g}_{\kappa})^{SL_n}$, is isomorphic to
the algebra $A(GL_n)$.
\end{proof}

\subsection{} 
The {\em trigonometric Cherednik algebra $\sH^\trig_\kappa(PGL_n)$
of $PGL_n$-type} is defined as a subalgebra in $\sH^\trig_\kappa(GL_n)$ generated
by $\BC[W^a]$ and $\Sym(\bt)$. Equivalently, $\sH^\trig_\kappa(PGL_n)$
is generated by the subalgebras $\BC[W^a]$ and $\Sym(\bt)$ with
relations
\begin{equation}
\label{PGL1}
s_i\cdot \sy-s_i(\sy)\cdot s_i=
-\kappa\langle x_i-x_{i+1},\sy\rangle\
\forall \sy\in\bt,\ 1\leq i<n
\end{equation}
\begin{equation}
\label{PGL2}
[\sy,\bx]=\langle\eta,\sy\rangle\bx
-\kappa\sum_{1\leq i<j\leq n}\langle\alpha_{ij},\sy\rangle
\frac{\bx-s_{ij}(\bx)}{1-\bx_i^{-1}\bx_j}s_{ij},\
\forall \sy\in\bt,\ \bx=\exp(\eta),\ \eta\in P^0\subset\BC[W^a]
\end{equation}

The {\em trigonometric Cherednik algebra $\sH^\trig_\kappa(SL_n)$
of type $SL_n$} 
is generated by the subalgebras $\BC[W^e_0]$ and $\Sym(\bt)$ with
relations
\begin{equation}
\label{SL1}
s_i\cdot \sy-s_i(\sy)\cdot s_i=
-\kappa\langle x_i-x_{i+1},\sy\rangle\
\forall \sy\in\bt,\ 1\leq i<n
\end{equation}
\begin{equation}
\label{SL2}
\omega\cdot \sy=\omega(\sy)\cdot\omega\
\forall \sy\in\bt,\ \omega\in\Omega
\end{equation}
\begin{equation}
\label{SL3}
[\sy,\bx]=\langle\eta,\sy\rangle\bx
-\kappa\sum_{1\leq i<j\leq n}\langle\alpha_{ij},\sy\rangle
\frac{\bx-s_{ij}(\bx)}{1-\bx_i^{-1}\bx_j}s_{ij},\
\forall \sy\in\bt,\ \bx=\exp(\eta),\ \eta\in P_0\subset\BC[W^e_0]
\end{equation}
For an equivalent definition see e.g. ~\cite{Op},~3.6.

\bigskip

The algebras $\sH^\trig_\kappa(SL_n),\ \sH^\trig_\kappa(PGL_n),\
\sH^\trig_\kappa(GL_n)$ are closely related to each other.
To formulate the relations more precisely, let  $\D(\BC^\times)$ be
the algebra of differential operators  on $\BC^\times$. 
Let $\fx$ be a coordinate on $\BC^\times$, and $\fy=\fx\partial_\fx$.
Then $\D(\BC^\times)$ is generated by $\fx^{\pm1},\fy$ with the relation
$[\fy,\fx]=\fx$. 

We have an embedding $\BC^\times\hookrightarrow H$ as the scalar (central) 
matrices in $GL_n$. Taking product with the embedding 
$\BT\hookrightarrow H$, we obtain a morphism ($n$-fold covering)
$\BT\times\BC^\times\twoheadrightarrow H$. At the level of functions we
have an embedding $\BC[H]\hookrightarrow\BC[\BT]\otimes\BC[\BC^\times]$.
In coordinates, we have $\bx_i\mapsto\bx_i\otimes\bx_i,\ 1\leq i\leq n$.
The isogeny $\BT\times\BC^\times\twoheadrightarrow H$ induces an isomorphism
of the Lie algebras of our tori. The inverse isomorphism
$\h\to\bt\oplus\BC$ sends $\sy_i\in\h$ to
$(\sy_i-\frac{1}{n}\sum_{k=1}^n\sy_k,\ \frac{1}{n}\fy)$ (recall from the
previous paragraph that $\BC=\operatorname{Lie}\BC^\times$ is equipped with
the base $\{\fy\}$).

We consider the tensor product algebra
$\sH^\trig_\kappa(SL_n)\otimes\D(\BC^\times)$ together with the following 
elements:
$$\Xi(\bx_i):=\bx_i\otimes\fx;\quad
\Xi(\sy_i):=(\sy_i-\frac{1}{n}\sum_{k=1}^n\sy_k)\otimes1+
1\otimes\frac{1}{n}\fy;\quad
1\leq i\leq n.$$

The following is a straightforward corollary of definitions.

\begin{corollary}
\label{sledstvie}
\vi There is a natural isomorphism
$\sH^\trig_\kappa(SL_n)\simeq\sH^\trig_\kappa(PGL_n)\rtimes\Omega$.

\vii The map $\bx_i\mapsto\Xi(\bx_i),\ \sy_i\mapsto\Xi(\sy_i),\
s_{ij}\mapsto s_{ij};\ 1\leq i\ne j\leq n$, defines an injective
homomorphism $\Xi:\ \sH^\trig_\kappa(GL_n)\hookrightarrow
\sH^\trig_\kappa(SL_n)\otimes\D(\BC^\times)$.\qed
\end{corollary}



\begin{corollary}\label{morita_cor} For any $c\in \C\sminus [-1,0)$,
  the functor below is a Morita equivalence
$$\Lmod{\sH^\trig_\kappa(SL_n)}\map
\Lmod{\e\sH^\trig_\kappa(SL_n)\e},\en
M\mto \e M.$$
\end{corollary} 
\begin{proof} 
We need to prove that $\sH^\trig_\kappa(SL_n)=
\sH^\trig_\kappa(SL_n)\cd\e\cd\sH^\trig_\kappa(SL_n)$.
The similar statement for the trigonometric Cherednik algebra
$\sH^\trig(GL_n)$ follows from Proposition~\ref{morita} 
(for $C=\BC^\times$).
Therefore, we have an equality
\begin{equation}\label{1}
1=\sum_{k=1}^r u_k\cd\e\cd v_k,\quad u_1,\ldots,u_r,v_1,\ldots,v_r\in
\sH^\trig_\kappa(GL_n).
\end{equation}
This implies, via the  imbedding
$\sH^\trig_\kappa(GL_n)\into
\sH^\trig_\kappa(SL_n)\otimes\D(\BC^\times),$ of
Corollary \ref{sledstvie}(ii),
that a similar equality holds in the algebra
$\sH^\trig_\kappa(SL_n)\otimes\D(\BC^\times)$.
Thus, in  $\sH^\trig_\kappa(SL_n)\otimes\D(\BC^\times)$ we have 
\begin{equation}\label{11}
1\o 1=\sum_{k=1}^r (u_k\cd\e\cd v_k)\o w_k,\quad
u_k,v_k\in \sH^\trig_\kappa(SL_n),\,
w_k\in \D(\BC^\times).
\end{equation}

 Choose a $\C$-linear countable basis
$\{q_\nu\}_{\nu\in \mathbb N}$ 
of the vector space $\D(\BC^\times)$ such that $q_1=1$.
Expanding each of the elements $w_k\in \D(\BC^\times)$
in this basis and equating the corresponding terms in \eqref{11}
one deduces from \eqref{11} an equation of the form
$1=\sum_\ell a_\ell\cdot\e\cdot b_\ell$, where
$a_\ell,b_\ell\in \sH^\trig_\kappa(SL_n).$
Thus, we have shown that 
$$\sH^\trig_\kappa(SL_n)=
\sH^\trig_\kappa(SL_n)\cd\e\cd\sH^\trig_\kappa(SL_n),
$$
and the Morita equivalence for the algebra $\sH^\trig_\kappa(SL_n)$
follows.
\end{proof}

\subsection{}
We have the $\syn$-equivariant product morphism $H=(\BC^\times)^n\to\BC^\times$
with the kernel $\BT\subset H$. We can consider the restriction of the sheaf of Cherednik
algebras $\sH_{\kappa}$ to the closed subvariety $\BT/\syn\sset H/\syn$.
Abusing the notation, we also write  $\sH_{\kappa}$ for the
corresponding algebra of global sections.

The proof of the following result copies the proof
of Proposition \ref{reduc GL_n}.

\begin{corollary}
\label{reduc SL_n}
The spherical trigonometric Cherednik subalgebra
$\e\sH_\kappa^\trig(SL_n)\e$ is isomorphic to the quantum Hamiltonian reduction
$(\D_{n\kappa}(SL_n\times\P)/\D_{n\kappa}
(SL_n\times\P){\mathfrak g}_{\kappa})^{SL_n}$.
\end{corollary}

An important difference between the algebras
$\e\sH_\kappa^\trig(GL_n)\e$ and $\e\sH_\kappa^\trig(SL_n)\e$
is that the latter may have finite dimensional representations,
while the former cannot have such representations.
In view of this, it is desirable to have a version of Hamiltonian
reduction functor that would relate $\D$-modules
on $SL_n\times \P$ (rather than
on $GL_n\times\P$) with $\e\sH_\kappa^\trig(SL_n)\e$-modules.
To this end, one needs  first to introduce a
Lagrangian nil-cone in $T^*(SL_n\times \P),$ and the corresponding
 notion of character $\D$-module
on  $SL_n\times \P$.

To define the nil-cone,
we intersect   $\Nnil(\BC^\times)\sset
GL_n\times\mathfrak{gl}_n\times\Vo\times
V^*=T^*(\rep_{\BC^\times}^n\times\Vo)$,
the nil-cone for the group $GL_n$, with
$SL_n\times\mathfrak{sl}_n\times\Vo\times V^*.$
Let $\Nnil\sset T^*(SL_n\times\P)$ 
be the Hamiltonian reduction
of the resulting variety with respect to the Hamiltonian
$\C^\times$-action on the factor $T^*\Vo=\Vo\times V^*.$

We write $\D_c:=\D(SL_n)\o \D_c(\P).$
An $SL_n$-equivariant $\D_c$-module is called character
$\D$-module provided its characteristic variety is contained in
$\Nnil$.

\subsection{A Springer type construction} We are going to introduce
certain analogues of Springer resolution in our present setting.
The constructions discussed below work more generally, in the framework of
an arbitrary smooth curve $C$. However, to simplify the exposition,
we restrict ourselves to the case $C=\C^\times$.

Write $\CB$ for the  flag variety, that is, the variety of complete flags 
$F=(0=F_0\sset F_1\sset\ldots \sset F_{n-1}\sset F_n=V),$ where
$\dim F_k=k,\forall k=0,\ldots,n.$
 In the trigonometric case, the variety
$\widetilde{\rep_C^n}$,
introduced in \S\ref{flags}, reduces to
$\wg=\{(g, F)\in SL_n\times \CB\;|\; g(F_k)\sset F_k,\ \forall
k=0,\ldots,n\}.$ The first projection yields a proper morphism
$\wg\onto SL_n$ known as the {\em Grothendieck-Springer resolution}.

Now, fix an integer $1\leq m\leq n$.
We define a closed subvariety $\X_{n,m}\sset \X_n$ as follows
$$
\X_{n,m}:=\{(g,\ell)\in SL_n\times \P\en|\en\dim (\C[g]\ell)\leq m\},
$$
where $\C[g]\ell\sset V$ denotes the minimal $g$-stable subspace in $V$
that contains the line $\ell\sset V$.

Next, we put $\wt
\X_n=\wg\times \P$ and
define a closed subvariety $\wt \X_{n,m}\sset \wt
\X_n$
as follows
$$\wt \X_{n,m}:=\{(g,F,\ell)\in SL_n\times \CB\times \P\en|\en
 \ell\in F_m,\en\&\en g(F_k)\sset F_k,\ \forall
k=0,\ldots,n\}.
$$
We have a diagram
$$
\xymatrix{
&&&\wt \X_{n,m}\ar[dlll]_<>(0.5){_{(g,F,\ell)\mto (g,\ell)}\en}
\ar[drrr]^<>(0.5){\en^{(g,F,\ell)\mto
(F,\ell)}}&&&\\
SL_n\times \P&&&&&&\CB\times \P.
}
$$

We observe that:
\vskip 3pt

\pb{The projection $\wt \X_{n,m}\to \CB\times \P,\, (g,F,\ell)\mto
(F,\ell)$
makes $\wt \X_{n,m}$ a locally trivial fibration  over
the base $\{(F,\ell)\in\CB\times \P\;|\; \ell\sset F_m\}.$
Both the fiber and the base are smooth.\newline
\hphantom{x}\qquad
Thus, $\wt \X_{n,m}$ is smooth.}
\vskip 2pt

\pb{The image of the projection $\wt \X_{n,m}\to SL_n\times \P$ is
contained
in $\X_{n,m}$, hence the projection gives a well defined  morphism:}
$$\pi_{n,m}:\
\wt \X_{n,m}\too \X_{n,m},\quad (g,F,v)\mto (g,v).
$$

Let $\Xr$ be  an open subset of $\X_{n,m}$ formed by the
pairs $(g,\ell)$ such that $g$ is a matrix with $n$ pairwise
distinct eigenvalues. Let $\wtxr=\pi_{n,m}\inv(\Xr)$, an
open subset in $\wt \X_{n,m}$.

\begin{proposition}
\label{Springer}

\vi The map $\pi_{n,m}:\ \wt \X_{n,m}\to \X_{n,m}$  is a dominant
proper morphism,
which is 
{\em small} in the sense of Goresky-MacPherson.

\vii The restriction
$\pi_{n,m}: \wtxr\to\Xr$  is a Galois covering with the Galois
group ${\mathbb S}_m\times{\mathbb S}_{n-m}$.
\end{proposition}

\begin{proof} Clearly, $\pi_{n,m}$ is proper and has a dense image.
Hence,
it is dominant.
 
Let $Z:=\wt \X_{n,m}\times_{\X_{n,m}} \wt \X_{n,m}$, so we have a
projection
$Z\to \X_{n,m}.$
To prove that  $\pi_{n,m}$ is {\em semi}small, one
 must check that $\dim Z\leq n^2+m.$ To prove that
 $\pi_{n,m}$ is small one must show in addition
that each irreducible component of $Z,$ of dimension $n^2+m,$ dominates
$\X_{n,m}$. The argument is very similar to the standard proof of
 smallness of the Grothendieck-Springer resolution. 

In more detail, for $w\in{\mathbb S}_n$ we denote by
$Z_w$ the locally closed subvariety of
$Z$ formed by the quadruples $(g,F, F',\ell)$
 such that the flags $F$ and $F'$ are
in relative position $w$ (and such that $\ell\sset F_m\cap F'_m$, and 
$g(F)=F, \, g(F')=F'$). Then we have $Z=
\bigsqcup_{w\in{\mathbb S}_n}Z_w$.

We may view $Z_w$ as a fibration over an
$SL_n$-orbit in $\CB\times\CB,$ the cartesian square of flag variety.
We see immediately that
$\dim Z_w\leq n^2+m$ with an exact equality
if and only if $F_m=F'_m$. The latter equality holds if and only if 
$w\in{\mathbb S}_m\times{\mathbb S}_{n-m}\subset{\mathbb S}_n$.
For such a $w$ it is clear that $Z_w$ dominates
$\X_{n,m}$. This completes the proof of (i).

Now part (ii) follows easily from the above description of irreducible components.
\end{proof}

Let $\Sigma_N$ denote the set of partitions
of an integer $N$.
For any partition $\la\in\Sigma_m$, resp. $\mu\in\Sigma_{n-m}$, 
write $L_\la$, resp. $L_\mu$, for  an irreducible representation
of the Symmetric group ${\mathbb S}_m$, resp. ${\mathbb S}_{n-m}$,
associated with
that partition in a standard way. Thus,
$L_\la\boxtimes L_\mu$
is an irreducible representation
of the group ${\mathbb S}_m\times{\mathbb S}_{n-m}$.

Let $\C_{\wtxr}$ be the constant sheaf
on $ \wtxr$.
According to Proposition \ref{Springer}(ii), we have 
$$(\pi_{n,m})_*\C_{\wtxr}=
\bigoplus_{(\la,\mu)\in\Sigma_m\times\Sigma_{n-m}}
(L_\la\boxtimes L_\mu)\ \o \ {\scr L}_{\la,\mu},
$$
where ${\scr L}_{\la,\mu}$ is an irreducible local
system on $\Xr$ with monodromy $L_\la\boxtimes L_\mu$.

Now, let  $\C_{\wt\X_{n,m}}[\dim \wt \X_{n,m}]$ be the constant sheaf
on $\wt\X_{n,m}$, with shift normalization as a perverse sheaf.
Then, from part (i) of  Proposition \ref{Springer}, using the definition
of an intersection cohomology complex,
we deduce

\begin{corollary}
\label{principal series} There is a direct sum decomposition
$$
(\pi_{n,m})_*\C_{\wt\X_{n,m}}[\dim \wt \X_{n,m}]=
\bigoplus_{(\la,\mu)\in\Sigma_m\times\Sigma_{n-m}}
(L_\la\boxtimes L_\mu)\ \o \ IC({\scr L}_{\la,\mu}).\eqno\Box
$$
\end{corollary}
Here, $IC(-)$ denotes the intersection cohomology extension of a local
system.

One can also translate the statement of the corollary into 
a $\D$-module language. To this end, write
$i_{n,m}:\ \wt \X_{n,m}\hookrightarrow
\wt\X_n=\wg\times\P$ for an obvious closed embedding.
For any {\em integer} $c\in\Z$, one has a
$\D_{0,c}$-module 
$\CO(c)_{n,m}:=i_{n,m}^*\big(\CO_{\wg}\boxtimes\CO_\P(c)\big)[m-n],$
on $\wt\X_{n,m}$. 

Corollary \ref{principal series} yields the following result.

\begin{corollary} The direct image
$(\pi_{n,m})_*\CO(c)_{n,m}$ is a semisimple $\D_{0,c}$-module on
$\X_{n,m}$ and 
one has
 a direirect sum decomposition
$$
(\pi_{n,m})_*\CO(c)_{n,m}=
\bigoplus_{(\la,\mu)\in\Sigma_m\times\Sigma_{n-m}}
(L_\la\boxtimes L_\mu)\ \o \ \CF_{\lambda,\mu},
$$
where  $\CF_{\lambda,\mu}$ is
an irreducible character $\D_{0,c}$-module on $\X_n$.
\qed
\end{corollary}

\subsection{Cuspidal $\D$-modules.}\label{chetyre}
The goal of this subsection is to
describe  character
$\D_c$-modules which have finite dimensional Hamiltonian reduction.
These $\D_c$-modules turn out to be closely related
to {\em cuspidal} character sheaves on $SL_n$.

In more detail, write  $Z(SL_n)$ for the center
of the group $SL_n$. Thus,  $Z(SL_n)$
is a cyclic group, the group of scalar matrices of
the form $z\cdot\Id,$ where $z\in\C$ is an $n$-th root of
unity. 

Let $\bN\subset SL_n$ be the unipotent cone, and let
$\jmath: \bN^\reg\into\bN$ be an open imbedding of
the conjugacy class formed by the regular unipotent
elements. The fundamental group of $\bN^\reg$ may be identified
canonically with $Z(SL_n)$. For each  integer $p=0,1,\ldots,n-1,$ there 
 is a group homomorphism $Z(SL_n)\to\C^\times, \,z\cdot\Id\mto z^p.$
Let  $\sL_p$ be
 the corresponding rank one $SL_n$-equivariant local
system, on $\bN^\reg$, with monodromy $\theta=\exp(\frac{2\pi\sqrt{-1}p}{n})$.

From now on, we assume  that $(p,n)=1$, ie., that
$\theta$ is a primitive $n$-th root of
unity. Then, the  local
system $\sL_p$ is known to be {\em clean}, that is,
for $\D$-modules on $SL_n$,
one has $\jmath_!\sL_p\cong\jmath_{!*}\sL_p\cong\jmath_*\sL_p$,
cf. \cite{L} or \cite{Os}. Given a central element 
$z\in Z(SL_n)$, we have the conjugacy class  $z\bN^\reg\subset SL_n$, the $z$-translate of $\bN^\reg$,
and we let $z\jmath_!\sL_p$ denote  the corresponding translated 
 $\D$-module supported on the closure of
 $z\bN^\reg$. According to Lusztig \cite{L},  $z\jmath_!\sL_p$ is
a cuspidal character $\D$-module on the group $SL_n$.

Further, for any integer $c\in\Z,$ we may form
 $z\CL_{p,c}:=(z\jmath_!\sL_p)\boxtimes \CO(c)$,
a twisted $\D$-module on $SL_n\times\P$.
Thus, $z\CL_{p,c}$ is a simple  character
$\D_{c}$-module.

\begin{theorem} \label{cusp_thm}
\vi Let $c$ be a nonnegative real number
and let $\CF$ be a nonzero simple  character
$\D_{c}$-module. Then, the following properties are equivalent:
\begin{enumerate}
\item The support of $\CF$ is contained in $(Z(SL_n)\cdot\bN)\times \P;$
\item We have $\CF\cong z\CL_{p,c},$ for some $z\in Z(SL_n)$ 
and some integers $p,c$ such that $p$ is prime to $n$, and $0<p<n$.
\end{enumerate}

\vii For a simple character $\D_{c}$-module
 $\CF$, we have
\begin{equation}\label{trtr}
\BH(\CF)\neq 0\en\oper{and}\en
\dim \BH(\CF)<\infty
\quad\Longleftrightarrow\quad
\operatorname{(1)}\mbox{-}\operatorname{(2)}\en\oper{hold\ and,\
we\ have}\en
n|(c-p).
\end{equation}

\viii Let
$\kappa=c/n$, where $c$ and $n$ are mutually prime integers. Then, the 
functor $\BH$ yields a one-to-one correspondence between
 simple  character
$\D_{c}(SL_n\times\P)$-modules of the form $z\CL_{p,c}$, with $n|(c-p),$
and finite dimensional  irreducible
$\e\sH_\kappa^\trig(SL_n)\e$-modules, respectively.
\end{theorem}

\proof  
The implication $(2)\;\Rightarrow\;(1)$ of part (i) is clear.
We now prove that $(1)\;\Rightarrow\;(2)$.

The group $GL_n$ acts on $\bN\times\Vo$ with finitely many orbits,
see ~\cite{GG},~Corollary 2.2. 
The open orbit $\BO_0$ is formed by the pairs $(X,v)$
where $X\in\bN^\reg$, and $v\in V\setminus\im(X-1)$. 
The action of $GL_n$ on $\BO_0$
is free (and transitive). Let $\BO$ be a unique $GL_n$-orbit open in the
support of $\CF$. The singular support of $\CF$ contains the conormal
bundle $T^*_\BO(SL_n\times\Vo)$. According to~\cite{GG},~Theorem~4.3,
if $\BO$ lies in $\bN\times\Vo$, and $T^*_\BO(SL_n\times\Vo)$ lies in
$\mnil:=\Nnil(\BC^\times)\cap
T^*(SL_n\times\Vo)$, then $\BO=\BO_0$.

We see that $\CF$ must be the minimal extension of an
irreducible local system on $\BO_0$.
Since $\BO_0$ is a $GL_n$-torsor, its fundamental group is $\BZ$, and
an irreducible local system is 1-dimensional with monodromy $\vartheta$;
it is $GL_n$-monodromic with monodromy $\vartheta$. We will denote such
local system by $L_\vartheta$.
It remains to prove that
$\vartheta=\theta$. Note that the boundary of $\BO_0$ contains a
codimension 1 orbit $\BO_1$ formed by the pairs $(X,v)$ such that
$X\in\bN^\reg,\ v\in\im(X-1)\setminus\im(X-1)^2$. 
The monodromy of $L_\vartheta$
around $\BO_1$ equals $\vartheta^n$. Thus if $\vartheta^n\ne1$, the
singular support of the minimal extension of $L_\vartheta$ necessarily
contains $T^*_{\BO_1}(SL_n\times\Vo)$. As we have just seen, the latter
is not contained in $\Nnil$; a contradiction. Hence $\vartheta^n=1$.

Now the monodromy of $L_\vartheta$ along lines in $\Vo$ is also equal
to $\vartheta^n=1$. Hence $L_\vartheta$ is a pullback to 
$\bN^\reg\times\Vo$ of a local system on $\bN^\reg$ with monodromy
$\vartheta$. We conclude that the minimal extension of 
this latter local system must be a classical character sheaf on $SL_n$,
so $\vartheta$ must be a primitive root of unity $\theta$.
This completes the proof of part (i) of the theorem.

We prove part (ii).
It follows from \cite{CEE},~Theorem~9.19,
that  $\BH(\CL_{p,c})\neq 0$ iff the integer $c$ is prime to $n$, and
$p$ is the residue of $c$ modulo $n$.
Now the first statement
is immediate from Corollary \ref{findimcor}. 
Let $M$ be a finite dimensional
$\e\sH_\kappa^\trig(SL_n)\e$-module and let $\CF:={}^\top\BH(M)$.
Then, $\CF$ is a character $\D$-module by Proposition \ref{ham_fun}.
Thus, $\CF$ has finite length. Let $\CF_1,\ldots,\CF_r$
be the collection of simple subquotients of $\CF$,
counted with multiplicities. The inclusion in \eqref{sinv}
combined with Corollary \ref{findimcor}
imply that each of the $\D$-modules $\CF_i$ satisfies
the equivalent conditions (1)-(3) of part (i) of the theorem,
hence, has the form $\CF_i=z_i\CL_{p_i,c_i},$ for some integers $c_i$
prime to $n$ with residues $c_i$ modulo $n$, 
and some $z_i\in Z(SL_n)$. Therefore, as we have mentioned earlier,
for any $i=1,\ldots,r,$ one has $\BH(\CF_i)=\BH(z_i\CL_{p_i,c_i})\neq 0$.

On the other hand, the functor of Hamiltonian reduction is exact
and we know that $\BH({}^\top\BH(M))$ $=M,$  by Proposition \ref{ham_fun}.
We deduce that $\CF_i=0$ for all $i$ except one. 
Thus, $\CF$ is a simple character $\D$-module and
part (ii) follows. 
The statements of part (iii) are now clear.
\endproof
\setcounter{equation}{0}
\footnotesize{

\nopagebreak
\end{document}